\documentclass[10pt]{amsart}

\usepackage{amsmath,amsthm,amssymb, cancel, latexsym, mathrsfs, amsfonts, multirow}
\usepackage{graphics}
\usepackage{tikz}
\usepackage{subfig}
\usetikzlibrary{matrix,arrows}
\usepackage{verbatim}
\usepackage[paperwidth=9in, showframe=false, margin=1in]{geometry}

\usepackage[colorlinks=false,hyperindex]{hyperref}

\hypersetup{
	colorlinks,
	citecolor=violet,
	linkcolor=violet,
	urlcolor=blue}

\theoremstyle{Theorem}
\newtheorem{theorem}{Theorem}[section]
\newtheorem*{maintheorem}{Theorem}
\newtheorem{corollary}[theorem]{Corollary}
\newtheorem{lemma}[theorem]{Lemma}
\theoremstyle{definition}
\newtheorem{definition}[theorem]{Definition}
\newtheorem{example}[theorem]{Example}
\newtheorem{remark}[theorem]{Remark}

\newcommand{\C}[1]{\ensuremath{\mathcal{#1}}}
\newcommand{\B}[1]{\ensuremath{\mathbf{#1}}}

\newcommand{\FF}{\mathbf F}
\newcommand{\ZZ}{\mathbf Z}
\newcommand{\RR}{\mathbf R}
\newcommand{\QQ}{\mathbf Q}
\newcommand{\CC}{\mathbf C}
\newcommand{\NN}{\mathbf N}
\newcommand{\PP}{\mathbf P}
\newcommand{\fm}{\mathfrak m}
\newcommand{\fp}{\mathfrak p}
\newcommand{\mbf}{\mathbf}
\newcommand{\mcal}{\mathcal}
\newcommand{\Map}{\operatorname{Map}}
\newcommand{\Lip}{\operatorname{Lip}}
\newcommand{\ord}{\operatorname{ord}}

\def\W{\mathbf{W}}

\def\C{\mathcal{C}}

\def\cU{\mathcal{U}}
\newcommand{\PGL}{\textnormal{PGL}}

\newcommand{\SL}{\textnormal{SL}}
\newcommand{\Lev}{\textnormal{Lev}}

\def\Teich{Teichm\"{u}ller }

\def\tn{\textnormal}

\colorlet{darkgreen}{green!50!black}

\def\onto{\twoheadrightarrow}

\def\P1{\mathbf{P}^{1}(\QQ_p)}
\def\root{T_0}
\def\tree{\mathcal{T}}

\def\IV{\operatorname{IV}}

\author{Lance Edward Miller \and Benjamin Steinhurst}
\address[L.E.~Miller]{Department of Mathematical Sciences, University of Arkansas, Fayetteville, AR~72701} \email[L.E.~Miller]{lem016@uark.edu}
\address[B.~Steinhurst]{Department of Mathematics and Computer Science, McDaniel College, Westminster MD 21157}%, USA
\email[B.~Steinhurst]{bsteinhurst@mcdaniel.edu}

\title[Witt-Burnside rings and Lipschitz continuous functions]{Witt-Burnside functor attached to $\ZZ_p^2$ and $p$-adic Lipschitz continuous functions.}

\begin{document}

%\tableofcontents

\begin{abstract}
Dress and Siebeneicher gave a significant generalization of the construction of Witt vectors, by producing for any profinite group $G$, a ring-valued functor $\mathbf{W}_G$. This paper gives a concrete interpretation of the rings $\mathbf{W}_{\mathbf{Z}_p^2}(k)$ where $k$ is a field of characteristic $p > 0$ in terms of rings of Lipschitz continuous functions on the $p$-adic upper half plane $\mathbf{P}^1(\mathbf{Q}_p)$. As a consequence we show that the Krull dimensions of the rings $\mathbf{W}_{\mathbf{Z}_p^d}(k)$ are infinite for $d \geq 2$ and we show the Teichm\"uller representatives form an analogue of the van der Put basis  for continuous functions on $\mathbf{Z}_p$. 
\end{abstract}

\maketitle

\section{Introduction}

Rings of $p$-typical Witt vectors are ubiquitous in number theory and commutative algebra. A generalization due to Dress and Siebeneicher \cite{DS88} which introduces a ring valued functor $\W_G$ for each profinite group $G$ interpolates between many of the known generalizations, recovering the classical $p$-typical case with $G = \ZZ_p$ as an additive group and Cartier's `big' Witt vectors in the case $G = \widehat{\ZZ}$. We denote by $\W(k)$ the $p$-typical Witt vectors $\W_{\ZZ_p}(k)$ over a field $k$. These generalizations have been an important area of recent growing study \cite{Ell06, Mil, Mil13, OH1, OH2, OH12}. Despite this, very little is known about the image of $\W_G$ when $G$ is infinite besides the cases $G = \ZZ_p$ or $G = \widehat{\ZZ}$. In particular, besides these cases there are no known interpretations of the images of $\W_G$ in terms of other more familiar rings. The present article fills this gap in the case $G = \ZZ_p^2$. 

The first author has explored the image of $\W_G$ in the pro-$p$ cases $G = \ZZ_p^d$ when $d \geq 2$ \cite{Mil} and in the pro-dihedral case $G = \varprojlim D_{2^n}$ \cite{Mil13} over fields of characteristic $p$. There it is shown that the rings in the image are far from expected as compared to the more classic $p$-typical case as in these cases the image need not be a domain nor even noetherian (\cite[Thm. 4.5]{Mil} for the case $G = \ZZ_p^d$ with $d \geq 2$ and \cite[Thm. 4.6]{Mil13} for the case when $G =  \varprojlim D_{2^n}$.) One of the most important remaining areas of study is the nature of the prime spectrum. Using this as motivation, we give a concrete interpretation of a quotient of the ring $\W_{\ZZ_p^2}(k)$, when $k$ has characteristic $p$, as subring of the ring of continuous $p$-typical valued functions on $\PP^1(\QQ_p)$ which is sufficient to give a large source of prime ideals. 

This study starts by exploring a natural generalization of the usual topology on the $p$-typical Witt vectors to the rings $\W_{\ZZ_p^d}(k)$ where $d \geq 2$ and $k$ is a field of characteristic $p$, we call it the {\it initial vanishing topology}. This topology is based on a filtration of ideals introduced in \cite[Def. 2.12]{Mil} and is defined for any image of $\W_G$ where $G$ is any profinite group. The initial vanishing topology agrees with the natural profinite topology  in the case that $G$ is topologically finitely generated. For the $p$-typical Witt vectors (i.e., the case when $G = \ZZ_p$) these two topologies are the same as the topology defined by the maximal ideal.  In the case where $G$ is pro-$p$ and $k$ is a field of characteristic $p$ (or just local), the rings $\W_G(k)$ are also local (\cite[Thm. 2.16]{Mil}) and so inherit a topology coming from its maximal ideal which also can be thought of as generalizing the usual topology on the $p$-typical Witt vectors. We show these two topologies differ in the case $G = \ZZ_p^d$ when $d \geq 2$.

\begin{maintheorem} $($Corollary~\ref{cor:TopNotSame}$)$
For $k$ a field of characteristic $p$, the initial vanishing topology and the topology defined by the maximal ideal in $\W_{\ZZ_p^d}(k)$ for $d \geq 2$ do not agree. 
\end{maintheorem}

There is a particular technical requirement (called {\it the ratio property}, see \ref{dff:ratio}) which is always satisfied when $G$ is abelian, and allows for a norm generating the initial vanishing topology. This is particularly interesting in the case that $G = \ZZ_p^2$, where we determined that $p$ is not a zero divisor in $\W_{\ZZ_p^2}(k)$ where $k$ is a field of characteristic $p > 0$, and so this norm extends uniquely to give a $\QQ_p$-algebra norm on $\W_{\ZZ_p^2}(k)[{1 \over p}]$, see Section~\ref{sec:norms}. 

When $k$ is a field of characteristic $p$, $\W_{\ZZ_p^2}(k)$ is also reduced \cite[Thm. 5.17]{Mil}. This shows immediately that a certain natural family of prime ideals defined in Section~\ref{sec:PI} are provably not exhaustive of prime ideals in $\W_{\ZZ_p^2}(k)$. The following identification of $\W_{\ZZ_p^2}(k)$ as a ring of functions allows us to construct many more. 

\begin{maintheorem}$($Theorem~\ref{thm:isomorphism}$)$
%\label{thm:isomorphism} 
There is a homomorphism from $$\W_{\ZZ_p^2}(k) \to \C(\PP^1(\QQ_p),\W(k)),$$ where $\C(\PP^1(\QQ_p),\W(k))$ is the ring of continuous functions from $\PP^1(\QQ_p)$ to the $p$-typical Witt vectors $\W(k)$.  
\end{maintheorem}

%The more precise version of this theorem is stated as Theorem~\ref{thm:isomorphism} where we also pin down exactly the kernel and image of $\Phi$, making it an isomorphism. This identification we can utilize techniques of studying prime ideals in function spaces to prove the following as an application. 

This identification can be utilized to study prime ideals of $\W_{\ZZ_p^2}(k)$ via techniques used to study prime ideals of $\C(\PP^1(\QQ_p),\W(k))$. An application of this approach is the following.

\begin{maintheorem}$($Theorem~\ref{thm:diminfinite}$)$
%\label{thm:diminfinite}
The Krull dimension of $\W_{\ZZ_p^d}(k)$ is infinite for $d \geq 2$. 
\end{maintheorem}

In determining the image of the homomorphism in the main theorem the metric structure of these spaces comes into play. In particular the image is precisely $\Lip_{p^{-1}}(\PP^1(\QQ_p),\W(k))$. Since both spaces ($\W_{\ZZ_p^d}(k)$ and $\Lip_{p^{-1}}(\PP^1(\QQ_p),\W(k))$) involved in this homomorphism are metric spaces we consider the analytic properties. It is noteworthy that the natural Lipschitz space metric on $\Lip_{p^{-1}}(\PP^1(\QQ_p),\W(k))$ is not the correct metric for our applications, instead we consider $\Lip_{p^{-1}}(\PP^1(\QQ_p),\W(k))$ as a subring of the metric space of continuous functions with the supremum norm.  Doing so, we can show that after factoring out the kernel, the induced map is also isometric, Theorem~\ref{thm:Isometry}. We give an analytic characterization of the image of the natural generalization of the \Teich elements for $\W(k)$ to $\W_{\ZZ_p^2}(k)$ which form a topological basis in $\W_{\ZZ_p^2}(k)$ in its initial vanishing topology.

\begin{maintheorem}$($Theorem~\ref{thm:TeichvdP}$)$
%\label{thm:TeichvdP}
The image of the \Teich basis for $\W_{\ZZ_p^2}(k)$ form a scaled analogue of a van der Put basis for $\Lip_{p^{-1}}(\PP^1(\QQ_p),\W(k))$. 
%Denote by $B \subset \Lip_{p^{-1}}(\PP^1(\QQ_p),\W(k))$ the image under $\Phi$ of the Teichm\"uller basis for $R$. 
%The set $B$ consists of elements of the{which? I think what you were trying to say in the earlier draft is that there is an analogue vdP for something like $\C(\ZZ_p, \W(k))$. This is what we should describe precisely before this theorem and state it as what you mean here. } van der Put basis for $C(\PP^1(\QQ_p),\W(k))$ scaled to have Lipschitz constant $1/p$. 
\end{maintheorem}

Throughout we assume the reader is familiar with the $p$-typical Witt vectors. To ease exposition, we include a more elaborate review of Witt-Burnside rings as Section~\ref{sec:WB}. This section is mostly expository and as self-contained as possible, including a restatement of a few key results in \cite{Mil}. We also review some basic information on function rings on ultrametric spaces in Section~\ref{sec:FunRing}. 

{\it Acknowledgments:} The authors would like to thank Roi Docampo-\'{A}lvarez, Paul Roberts, Keith Conrad, Kiran Kedlaya, Veronica Ertl, and Jim Stankewicz for many helpful discussions about the paper.

\section{Witt-Burnside rings}\label{sec:WB}

Witt-Burnside rings are constructed utilizing generalized Witt polynomials associated to a profinite group $G$. They were introduced by Dress and Seibeneicher \cite{DS88}. We refer the reader to this or \cite{Mil} for a more elaborate introduction but give a brief review now. The index set of these generalized polynomials is the set of isomorphism classes of discrete finite transitive $G$-sets called the {\it{frame}} of $G$ and denoted $\mcal{F}(G)$. There is a natural partial ordering on $\mcal{F}(G)$. For $T$ and $U$ in $\mcal{F}(G)$ we say $U \leq T$ if there is a $G$-map from $T$ to $U$. Denote the set of all $G$-maps from $T$ to $U$ as $\Map_G(T,U)$ and the number of $G$-maps $\# \Map_G(T,U)$ by $\varphi_T(U)$. Thus $\varphi_T(U) \neq 0$ if and only if $T \leq U$. We summarize some facts about $\mcal{F}(G)$. %which will be proved later (Lemma \ref{lem:fiber}).

\begin{enumerate}
\item If $T$ and $U$ in $\mcal{F}(G)$ with $U \leq T$, then $\# U$ divides $\# T$ and $\# T / \# U$ represents the size of any of the fibers of any element of $\Map_G(T,U)$. 
\item If the stabilizer subgroups of the points in $T$ are all equal (we will say in this case that $T$ has normal stabilizers or that $T$ is a normal $G$-set), then $\varphi_T(U) = \# U$ for $U \leq T$. 
\item For each $T$ in $\mcal{F}(G)$, there are only finitely many $U$ in $\mcal{F}(G)$ with $U \leq T$. 
\end{enumerate} The elements of $\mcal{F}(G)$ have a concrete description. Every finite transitive $G$-set $T$ is isomorphic to some coset space $G/H$ with left $G$-action, where $H$ is an open subgroup of $G$ that can be chosen as the stabilizer subgroup of any point in $T$.  
The partial order $\leq$ on coset spaces (considered as $G$-sets up to isomorphism) can be described concretely by $G/K \leq G/H$ if and only if $H$ is conjugate to a subgroup of $K$ (or equivalently, $H$ is a subgroup of a conjugate of $K$). 

For $T \in \mcal{F}(G)$, 
define the $T$-th {\it Witt polynomial} to be 
\begin{equation}\label{Wdef}
W_T(\{X_U\}_{U \in \mcal{F}(G)}) = \sum\limits_{U \leq T} \varphi_T(U) X_U^{\#T / \#U} = X_0^{\#T} + \ldots + \varphi_T(T) X_T,
\end{equation}
where $0$ denotes the trivial $G$-set $G/G$. Trivially $\varphi_T(0) = 1$ for all $T$ in $\mcal{F}(G)$. 
This is a finite sum since there are only finitely many 
$U \leq T$.

For instance, if $G = \mbf{Z}_p$ then the finite transitive $G$-sets up to isomorphism 
are $\mbf{Z}_p/p^n\mbf{Z}_p$ for $n \geq 0$ and the Witt polynomial associated to $\mbf{Z}_p/p^n\mbf{Z}_p$ is the classical $n$-th $p$-typical Witt polynomial. Figure \ref{fig:frame} displays the frame of $\ZZ_2^2$. When $G = \ZZ_p^2$, the primary case of study in this article, all $G$-sets have $p+1$ covers (that is, there are exactly $p+1$ $G$-sets in the frame lying immediately above each $G$-set). Other than the trivial $G$-set, each $G$-set below the horizontal line in Figure \ref{fig:frame} has $p$ covers also below the horizontal line. This line is not part of the frame and depicts a devision among types of $\ZZ_p^d$ sets based on their {\it level} as defined in \cite[Def. 5.1]{Mil}. 

\begin{figure}[htb]
\begin{center}
\scalebox{0.8}{
\begin{tikzpicture}[smooth]

% Draw axes

\fill[black] (0,0) circle (0.1cm);

\fill[black] (1,0) circle (0.1cm);
\fill[black] (1,-1) circle (0.1cm);
\fill[black] (1,1) circle (0.1cm);

\fill[black] (3,1.25) circle (0.1cm);
\fill[black] (3,0.75) circle (0.1cm);

\fill[black] (3,0.25) circle (0.1cm);
\fill[black] (3,-0.25) circle (0.1cm);

\fill[black] (3,-0.75) circle (0.1cm);
\fill[black] (3,-1.25) circle (0.1cm);

\fill[black] (3,2.5) circle (0.1cm);

\fill[black] (5,3) circle (0.1cm);
\fill[black] (5,2.5) circle (0.1cm);
\fill[black] (5,2) circle (0.1cm);

\fill[black] (5,-1.2) circle (0.1cm);
\fill[black] (5,-1.45) circle (0.1cm);

\fill[black] (5,-0.65) circle (0.1cm);
\fill[black] (5,-0.9) circle (0.1cm);

\fill[black] (5,-0.15) circle (0.1cm);
\fill[black] (5,-0.4) circle (0.1cm);

\fill[black] (5,0.35) circle (0.1cm);
\fill[black] (5,0.1) circle (0.1cm);

\fill[black] (5,0.85) circle (0.1cm);
\fill[black] (5,0.6) circle (0.1cm);

\fill[black] (5,1.35) circle (0.1cm);
\fill[black] (5,1.1) circle (0.1cm);

\fill[black] (7,0) circle (0.05cm);
\fill[black] (7.2,0) circle (0.05cm);
\fill[black] (7.4,0) circle (0.05cm);

\fill[black] (5.5,2.5) circle (0.05cm);
\fill[black] (5.7,2.5) circle (0.05cm);
\fill[black] (5.9,2.5) circle (0.05cm);

\fill[black] (4,3.25) circle (0.05cm);
\fill[black] (4.2,3.45) circle (0.05cm);
\fill[black] (4.4,3.65) circle (0.05cm);
%\fill[black] (5.7,2.5) circle (0.05cm);
%\fill[black] (5.9,2.5) circle (0.05cm);

\draw[-] (0,0) -- (1,0);
\draw[-] (0,0) -- (1,-1);
\draw[-] (0,0) -- (1,1);

\draw[-] (1,1) -- (3,1.25);
\draw[-] (1,1) -- (3,0.75);

\draw[-] (1,0) -- (3,0.25);
\draw[-] (1,0) -- (3,-0.25);

\draw[-] (1,-1) -- (3,-0.75);
\draw[-] (1,-1) -- (3,-1.25);

\draw[-] (-1,1.5) -- (6,1.5);

\draw[-] (1,1) -- (3,2.5);
\draw[-] (1,0) -- (3,2.5);
\draw[-] (1,-1) -- (3,2.5);

\draw[-] (3,2.5) -- (5,3);
\draw[-] (3,2.5) -- (5,2.5);
\draw[-] (3,2.5) -- (5,2);

\draw[-] (3,1.25) -- (5,1.35);
\draw[-] (3,1.25) -- (5,1.1);

\draw[-] (3,0.75) -- (5,0.85);
\draw[-] (3,0.75) -- (5,0.6);

\draw[-] (3,0.25) -- (5,0.35);
\draw[-] (3,0.25) -- (5,0.1);

\draw[-] (3,-0.25) -- (5,-0.15);
\draw[-] (3,-0.25) -- (5,-0.4);

\draw[-] (3,-0.75) -- (5,-0.65);
\draw[-] (3,-0.75) -- (5,-0.9);

\draw[-] (3,-1.25) -- (5,-1.2);
\draw[-] (3,-1.25) -- (5,-1.45);

\draw[-] (3,1.25) -- (5,3);
\draw[-] (3,0.75) -- (5,3);

\draw[-] (3,0.25) -- (5,2.5);
\draw[-] (3,-0.25) -- (5,2.5);

\draw[-] (3,-0.75) -- (5,2);
\draw[-] (3,-1.25) -- (5,2);

\node at (-0.75,0) {$\ZZ_2^2/\ZZ_2^2$} {};
\node at (2,2.5) {$\ZZ_2^2/2 \ZZ_2^2$} {};

%\draw (5,3) circle (0.15cm);
%\draw (5,2.5) circle (0.15cm);
%
%\draw (5,-1.2) circle (0.15cm);
%\draw (5,-1.45) circle (0.15cm);
%
%\node at (5.4,-1.15) {$T_1$} {};
%\node at (5.4,-1.5) {$T_2$} {};
%
%\node at (5.4,3) {$V_1$} {};
%\node at (5.4,2.2) {$V_2$} {};

\end{tikzpicture}
}
\end{center}
\caption{The frame $\mcal{F}(\ZZ_2^2)$.}
\label{fig:frame}
\end{figure}
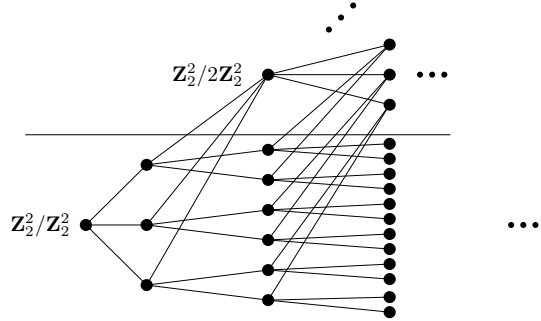

%\begin{figure}[h]
%\begin{center}
%\scalebox{0.3}{\includegraphics{FrameZ2sq.pdf}}
%\end{center}
%\caption{The frame $\mcal{F}(\ZZ_2^2)$.}
%\label{fig:frame}
%\end{figure}

\begin{remark}
\label{rmk:trees}
The picture in Figure \ref{fig:frame} is reminiscent of the tree 
of $\ZZ_2$-lattices in $\QQ_2^2$ up to $\ZZ_2$-scaling, on which $\PGL_2(\QQ_2)$ acts \cite[p.~71]{Ser80}; i.e., the Bruhat-Tits building for $\SL_2(\QQ_p)$. However, it is different since the $\ZZ_2$-sets $\ZZ_2^2/2^r\ZZ_2^2$ for different $r$ appear as separate vertices in Figure \ref{fig:frame}, while the subgroups $2^r\ZZ_2^2$ all correspond to the same vertex in the tree for $\PGL_2(\QQ_2)$. We exploit the relationship between $\mcal{F}(\ZZ_p^2)$ and the tree described in \cite{Ser80} to prove the theorems in the introduction. 
\end{remark}

To simplify notation, write a tuple of variables $X_T$ indexed by all $T$ in $\mathcal{F}(G)$ 
as $\underline{X}$, e.g.,  
$$W_T(\{X_U\}_{U \in \mcal{F}(G)}) = W_T(\underline{X}), 
\ZZ[\{X_T\}_{T \in \mcal{F}(G)}] = \ZZ[\underline{X}],$$ and $\ZZ[\{X_T, Y_T\}_{T \in \mcal{F}(G)}] = \ZZ[\underline{X}, \underline{Y}]$.  
This underline notation of course depends on $G$ but is surpressed.  For any commutative ring $A$, 
a polynomial 
$f(\underline{X}) \in \ZZ[\underline{X}]$ 
defines a function from $\prod_{T \in \mcal{F}(G)} A$ to $A$. For a tuple $\mbf{a} = (a_T)_{T \in \mcal{F}(G)}$ with each $a_T \in A$ we write  
$f(\mbf{a})$ for the value $f(\{a_T\}_{T \in \mcal{F}(G)}) \in A$.  A similar meaning is 
applied to $f(\mbf{a},\mbf{b})$ for a polynomial 
$f(\underline{X},\underline{Y}) \in \ZZ[\underline{X},\underline{Y}]$.  Generally, we write sequences indexed by $\mcal{F}(G)$ as bold letters $(e.g., \mbf{a},\mbf{b},\mbf{x},\mbf{y},\mbf{v})$ and their $T$-th coordinate is in italics $(e.g., a_T,b_T,x_T,y_T,v_T)$.

Because $X_T$ appears on the right side of (\ref{Wdef}) just in the linear term $\varphi_T(T)X_T$, 
and all variables which appear in other terms are $X_U$ for $U < T$, we get 
the following uniqueness criterion for all the Witt polynomial values together which is equivalent to Lemma 2.1 in \cite[p. 331]{Ell06}.

\begin{theorem}\label{invertthm}
If $A$ is a commutative ring which has no $\varphi_T(T)$-torsion, 
then the function $\prod_{T \in \mcal{F}(G)} A \rightarrow 
\prod_{T \in \mcal{F}(G)} A$ given by 
$\mbf{a} \mapsto (W_T(\mbf{a}))_{T \in \mcal{F}(G)}$ is injective. This function is bijective provided each $\varphi_T(T)$ is a unit in $A$. 
\end{theorem}

Applying Theorem \ref{invertthm} to the ring 
$\QQ[\underline{X},\underline{Y}]$ and the vectors 
$(W_T(\underline{X}) + W_T(\underline{Y}))_{T \in \mcal{F}(G)}$ and 
$(W_T(\underline{X})W_T(\underline{Y}))_{T \in \mcal{F}(G)}$
one obtains unique families of polynomials $\{S_T(\underline{X},\underline{Y})\}$ and 
$\{M_T(\underline{X},\underline{Y})\}$ in 
$\QQ[\underline{X},\underline{Y}]$ satisfying 
$$
W_T(\underline{X}) + W_T(\underline{Y}) = W_T(\underline{S}) \text{ for all } T \in \mathcal{F}(G) \text{ and } 
$$
$$
W_T(\underline{X})W_T(\underline{Y}) = W_T(\underline{M}) \text{ for all } T \in \mathcal{F}(G).
$$
More explicitly, this says 
\begin{equation}\label{STF}
\sum_{U \leq T} \varphi_T(U) X_U^{\# T / \#U} + \sum_{U \leq T} \varphi_T(U) Y_U^{\# T / \#U} = 
\sum_{U \leq T} \varphi_T(U) S_U^{\# T / \#U} \text{ and } 
\end{equation}

\begin{equation}\label{MTF}
\left(\sum_{U \leq T} \varphi_T(U) X_U^{\# T / \#U}\right)\left(\sum_{U \leq 
T} \varphi_T(U) Y_U^{\# T / \#U}\right) = 
\sum_{U \leq T} \varphi_T(U) M_U^{\# T / \#U}
\end{equation}
for all $T$.  
The polynomials $S_T$ and $M_T$ each only 
depend on the variables $X_U$ and $Y_U$ for $U \leq T$. 

A significant theorem of Dress and Siebeneicher \cite[p.~107]{DS88}, which generalizes Witt's theorem 
($G = \ZZ_p$), says that
the polynomials $S_T$ and $M_T$ have coefficients in $\ZZ$. 
We call the $S_T$'s and $M_T$'s the Witt addition and multiplication polynomials, respectively. 
(Obviously they depend on $G$, but that dependence will not be part of the notation). 

\begin{example}
\label{xmp:MTST}
Taking $T = 0$, one has 
$$
S_0(\underline{X},\underline{Y}) = X_0 + Y_0 \tn{ and} \
M_0(\underline{X},\underline{Y}) = X_0Y_0.
$$  
If $T \cong G/H$ where $H$ is a maximal open 
subgroup, so $\{U \in \mcal{F}(G) \colon U \leq T\}$ is just $\{0,T\}$. Solving for $S_T$ and $M_T$ in 
(\ref{STF}) and (\ref{MTF}) yields
$$
S_T = X_T + Y_T + 
\frac{(X_0+Y_0)^{\#T} - X_0^{\#T} - Y_0^{\#T}}{\varphi_T(T)},$$
$$
M_T = X_0^{\#T}Y_T + X_TY_0^{\#T} + \varphi_T(T)X_TY_T. 
$$
\end{example}

%\begin{remark}
%If a finite group $G$ contains a maximal subgroup $H$ of index $4$ 
%(e.g., $G = A_4$ and $H = A_3$), 
%$H$ can't be a normal subgroup of $G$ since otherwise the $G$-set 
%$T = G/H$ has $\varphi_T(T) = \#T = 4$ and 
%$((X_0+Y_0)^4 - X_0^4 - Y_0^4)/4$ doesn't have all integral coefficients.
%Of course it can be proved by group theory alone that
%$H$ isn't normal:  if $H \lhd G$ then $G/H$ is a group of order 4 and 
%such a group contains a subgroup of order 2, which contradicts maximality of 
%$H$ in $G$. 
%\end{remark}

Compare these with the first two $p$-typical Witt addition and multiplication 
polynomials.
Generally, the addition and multiplication 
polynomials, even in the $p$-typical case, are prohibitively complicated to write out explicitly for $\# T$ large.
Since $S_T$ and $M_T$ have integral coefficients, they can be evaluated on any ring, including 
rings where the hypotheses of Theorem \ref{invertthm} break down, like a ring of characteristic $p$ 
when $G$ is a pro-$p$ group. % (Example \ref{xmp1}).

\begin{definition}
Let $G$ be a profinite group. 
For any commutative ring $A$, define the {\it Witt--Burnside ring} $\mbf{W}_G(A)$ to be 
$\prod_{T \in \mathcal{F}(G)} A$ as a set, with 
elements written as $\mbf{a} = (a_T)_{T \in \mathcal{F}(G)}$. 
Ring operations on $\mbf{W}_G(A)$ are 
defined using the Witt addition and multiplication polynomials:
$$
\mbf{a} + \mbf{b} = (S_T(\mbf{a},\mbf{b}))_{T \in \mathcal{F}(G)} \text{ and }
$$
$$
\mbf{a} \cdot \mbf{b} = (M_T(\mbf{a},\mbf{b}))_{T \in \mathcal{F}(G)}.
$$
The additive (resp. multiplicative) identity is $(0,0,0,\dots)$ (resp. $(1,0,0,\dots)$). 
\end{definition}

One typically proves an algebraic identity in $\W_G(A)$ by reformulating it 
as an identity in a ring of Witt vectors over a polynomial ring over $\ZZ$. In this article, when needed, we only demonstrate the reformulation but may not go through the deduction of the identity we want over $A$ from the identity proved over a 
polynomial ring; instead we simply invoke functoriality. 

Even when $G$ is not abelian $\W_G(A)$ is a commutative ring. For any homomorphism of commutative rings $f \colon A \rightarrow B$ 
define $\W_G(f) \colon \W_G(A) \rightarrow \W_G(B)$ 
by applying $f$ to the coordinates: 
$$
\W_G(f)(\mbf{a}) = (f(a_T))_{T \in \mcal{F}(G)} \in \W_G(B).
$$
%Because the polynomials $S_T$ and $M_T$ 
%have integral coefficients, $\W_G(f)$ is a ring homomorphism 
%and composition of ring homomorphisms is respected,  
This is a ring homomorphism and makes $\W_G$ a covariant functor from commutative rings to commutative rings.

%We previously saw that each Witt polynomial $W_T$ defines a function 
%$\W_G(A) \rightarrow A$ such that 
%$$
%W_T(\mbf{a} \oplus \mbf{b}) = W_T(\mbf{a}) + W_T(\mbf{b}) \text{ and } 
%W_T(\mbf{a} \odot \mbf{b}) = W_T(\mbf{a})W_T(\mbf{b}).
%$$
%Since $W_T(1,0,0,\dots) = 1$, $W_T \colon \W_G(A) \rightarrow A$ is a ring homomorphism.  That Witt polynomials 
%define ring homomorphisms out of $\W_G(A)$ will be very useful to us later.

Packaging all the Witt polynomials together, we get a ring homomorphism 
$W : \W_G(A) \to \prod_{T \in \mcal{F}(G)} A$ which is $W_T$ in the $T$-th coordinate:
$$W(\mbf{a}) = (W_T(\mbf{a}))_{T \in \mathcal{F}(G)} = \left( \sum_{U \leq T} \varphi_T(U) a_U^{\# T/ \# U} \right)_{T \in \mcal{F}(G)}.$$
This homomorphism is called the {\it ghost map} and its coordinates 
$W_T(\mbf{a})$ are called the {\it ghost components} of $\mbf{a}$.    
If $A$ fits the hypothesis of Theorem \ref{invertthm} then 
the ghost map is injective (i.e., the ghost components of $\mbf{a}$ determine $\mbf{a}$). 
Also from Theorem \ref{invertthm} the ghost map is bijective if every integer $\varphi_T(T)$ is invertible in $A$ so $\W_G(A) \cong \prod_{T \in \mcal{F}(G)} A$ by the ghost map.
That means $\W_G(A)$ is a new kind of ring only if 
some $\varphi_T(T)$ is not invertible in $A$, and especially 
if $A$ has $\varphi_T(T)$-torsion for some $T$ (e.g., $G$ is a nontrivial pro-$p$ group and $A$ has characteristic $p$). 

The coordinates on which a Witt vector is nonzero is called its {\it support}. While the ring operations in $\mbf{W}_G(A)$ are generally not componentwise, addition in $\mbf{W}_G(A)$ is 
componentwise on two Witt vectors with disjoint support.

\begin{theorem}$($c.f. \cite[Thm. 2.6]{Mil}$)$
\label{thmRS}
Let $\{R,S\}$ be a partition of $\mcal{F}(G)$, i.e., $R \cup S = \mathcal{F}(G)$ and $R \cap S = \emptyset$. 
For every ring $A$ and any $\mbf{a} \in \mbf{W}_G(A)$, define $\mbf{r}(\mbf{a})$ and 
$\mbf{s}(\mbf{a})$ to be the Witt vectors derived from $\mbf{a}$ with 
support in $R$ and $S$: 
\begin{displaymath}
\mbf{r}(\mbf{a}) = 
\begin{cases}
a_T & \text{if $T \in R,$} \\
0 & \textrm{if $T \in S$,}
\end{cases} \quad \text{ and } \quad 
\mbf{s}(\mbf{a}) = 
\begin{cases}
0 & \text{if $T \in R,$} \\
a_T & \textrm{if $T \in S$.}
\end{cases}
\end{displaymath} Then $\mbf{a} = \mbf{r}(\mbf{a}) + \mbf{s}(\mbf{a})$ in $\mbf{W}_G(A)$. 
\end{theorem}

We also utilize the following theorem about the structure of the sum and product polynomials. 

\begin{theorem}$($c.f., \cite[Thm. 2.9]{Mil} $)$
\label{thm:integralwitt}
Give the ring $\mbf{Z}[\underline{X},\underline{Y}]$ the grading in which the degree of $X_U$ and $Y_U$ is $\# U$. 

\begin{enumerate}
\item[$(a)$] 
For all $T$, the polynomial $S_T$ is homogeneous of degree $\# T$ and $M_T$ is homogeneous of degree $2\#T$. 

\item[$(b)$]  For all $T$, $S_T(\underline{X},\mbf{0}) = X_T$, 
$S_T(\mbf{0},\underline{Y}) = Y_T$, 
$M_T(\underline{X},\mbf{0}) = 0$, and $M_T(\mbf{0},\underline{Y}) = 0$.
\end{enumerate}
\end{theorem}

Throughout we utilize three notions about $\ZZ_p^d$-sets for $d \geq 2$ some of which do not exist when $d = 1$. In particular, the algebraic results in \cite{Mil} depended on a careful study of both {\it cyclic} $\ZZ_p^d$-sets (\cite[Def. 3.1]{Mil}) which makes sense more generally for any profinite $G$, and the concepts {\it linked pairs of $\ZZ_p^d$} sets (\cite[Def. 3.3]{Mil}) and of the {\it level of a $\ZZ_p^d$} set (\cite[Def. 5.1]{Mil}) which both make sense for any $d \geq 2$. We denote by $\Lev(T)$ the level of the $\ZZ_p^d$-set $T$. While the statements of the theorems in this paper do not require these definitions, these notions will be freely utilized in the proofs and the reader is referred to the corresponding definitions and sections of \cite{Mil} for more elaboration.

\subsection{Prime ideals}\label{sec:PI}

A  key algebraic result about the rings $\W_{\ZZ_p^2}(k)$ when $k$ is a field of characteristic $p > 0$ is that it is reduced \cite[Thm. 5.17]{Mil}. As such, the intersection of all prime ideals is $\{ \mbf{0} \}$, so we have an embedding $\W_{\ZZ_p^2}(k) \hookrightarrow \prod_\fp \W_{\ZZ_p^2}(k)/\fp$ by reduction mod $\fp$ for all prime ideals $\fp$. The latter ring is a product of domains. What are the rings $\W_{\ZZ_p^2}(k)/\fp$? 
%For $G = \ZZ_p^d$ we can describe a natural collection of prime ideals in $\W_G(k)$. 

Let $f \colon \ZZ_p^d \onto \ZZ_p$ be a continuous surjective group homomorphism. We will see in more detail in Section \ref{sec:Top} (or see \cite[Sec. 2.9]{DS88}), $f$ induces a surjective ring homomorphism $\W_{\ZZ_p^d}(k) \onto \W_{\ZZ_p}(k)$ whose kernel is a prime ideal. We want to make this ring homomorphism explicit in terms of the frame of $\ZZ_p^d$, so we can visualize all these prime ideals. 

Pick a $\ZZ_p$-basis $\{e_1,\ldots,e_d\}$ of $\ZZ_p^d$ so that $\ZZ_p^d = \sum_{i=1}^d \ZZ_p e_i$ and for some $d' \leq d$, $\ker f = \sum_{i=1}^{d'} \ZZ_p p^{a_i} e_i$ where $a_0,\ldots,a_{d'} \geq 0$. Then $\ZZ_p^d/ \ker f = \sum_{i=1}^{d'} (\ZZ_p / p^{a_i}\ZZ_p) \overline{e_i} + \sum_{i = d'+1}^d \ZZ_p \overline{e_i}$. Since we know this quotient is isomorphic to $\ZZ_p$, $d' = d-1$ and $a_1, \ldots, a_{d-1} = 0$, so $f(x_1e_1 + \ldots + x_de_d ) = f(e_d)x_d$ where $f(e_d) \in \ZZ_p^\times$. After readjusting our basis, we can assume $f(e_d) = 1$ so $f$ is projection onto the $e_d$-coordinate. Set $N = \ker f$. The open subgroups of $\ZZ_p^d$ containing $N$ are $N + \ZZ_p p^r e_d$ for $ r \geq 0$, so the image of $\mcal{F}(\ZZ_p^d/N)$ in $\mcal{F}(\ZZ_p^d)$ is the collection of $\ZZ_p^d$-sets \begin{equation}\label{eq:prime} \ZZ_p^d/(N+\ZZ_p p^re_d) = \sum_{i=1}^d \ZZ_p e_i/(\sum_{i=1}^{d-1} \ZZ_p e_i + \ZZ_p p^r e_d),\end{equation} which form a chain of cyclic $\ZZ_p^d$-sets of level $0$ in the sense of \cite[Def. 5.1]{Mil}. So as $r$ varies, the $\ZZ_p^2$-sets in (\ref{eq:prime}) form a path in level $0$ in $\mcal{F}(\ZZ_p^2)$. See Figure~\ref{fig:Primeideal} for a pictorial description where the colored nodes indicate the support of an arbitrary element in the kernel of the projection $\W_{\ZZ_2^2}(k) \onto \W_{\ZZ_2}(k)$. 

\begin{figure}[htb]
\begin{tikzpicture}[inner sep=0pt, scale=0.25pt]
\tikzstyle{dot}=[fill=black,circle,minimum size=4pt]
\tikzstyle{rdot}=[fill=red,circle,minimum size=4pt]
\tikzstyle{bdot}=[fill=blue,circle,minimum size=4pt]
\tikzstyle{gdot}=[fill=darkgreen,circle,minimum size=4pt]
\tikzstyle{pdot}=[fill=magenta,circle,minimum size=4pt]

\node[pdot] (v0) at (0,0) {};

\node[dot] (v1) at (2,3.5) {};
\node[dot] (v2) at (2,0) {};
\node[pdot] (v3) at (2,-3.5) {};

\draw[-] (-3,5.5) -- (12,5.5);

\node[dot] (vL1) at (4,7) {};

\node[dot] (v4) at (4,4) {};
\node[dot] (v5) at (4,3) {};

\node[dot] (v6) at (4,0.5) {};
\node[dot] (v7) at (4,-0.5) {};

\node[pdot] (v8) at (4,-3) {};
\node[dot]  (v9) at (4,-4) {};

\node[dot] (v10) at (8,4.25) {};
\node[dot] (v11) at (8,5.25) {};

\node[dot] (v12) at (8,2.25) {};
\node[dot] (v13) at (8,3.25) {};

\node[dot] (v14) at (8,0.25) {};
\node[dot] (v15) at (8,1.25) {};

\node[dot] (v16) at (8,-0.25) {};
\node[dot] (v17) at (8,-1.25) {};

\node[dot] (v18) at (8,-2.25) {};
\node[pdot] (v19) at (8,-3.25) {};

\node[dot] (v20) at (8,-4.25) {};
\node[dot] (v21) at (8,-5.25) {};

\node[dot] (vL2) at (8,7) {};
\node[dot] (vL3) at (8,6) {};
\node[dot] (vL4) at (8,8) {};

\fill[black] (9,7) circle (0.15cm);
\fill[black] (9.5,7) circle (0.15cm);
\fill[black] (10,7) circle (0.15cm);

\fill[black] (9,0) circle (0.15cm);
\fill[black] (9.5,0) circle (0.15cm);
\fill[black] (10,0) circle (0.15cm);

\fill[black] (6,8) circle (0.15cm);
\fill[black] (6.5,8.5) circle (0.15cm);
\fill[black] (7,9) circle (0.15cm);

\draw (v0) -- (v1);
\draw (v0) -- (v2);
\draw[color=magenta] (v0) -- (v3);

\draw (v1) -- (v4);
\draw (v1) -- (v5);

\draw (v2) -- (v6);
\draw (v2) -- (v7);

\draw[color=magenta] (v3) -- (v8);
\draw (v3) -- (v9);

\draw (v4) -- (v10);
\draw (v4) -- (v11);

\draw (v5) -- (v12);
\draw (v5) -- (v13);

\draw (v6) -- (v14);
\draw (v6) -- (v15);

\draw (v7) -- (v16);
\draw (v7) -- (v17);

\draw (v8) -- (v18);
\draw [color=magenta](v8) -- (v19);

\draw (v9) -- (v20);
\draw (v9) -- (v21);

\draw (v1) -- (vL1);
\draw (v2) -- (vL1);
\draw (v3) -- (vL1);

\draw (vL1) -- (vL2);
\draw (vL1) -- (vL3);
\draw (vL1) -- (vL4);

%\path (0:0cm) node (v0) {$v_0$};
%
%\path (30:0.25cm) node (v1) {};
%\path (150:0.25cm) node (v2) {};
%\path (270:0.25cm) node (v3) {};
%
%
%\path (10:0.5cm) node (v4) {};
%\path (50:0.5cm) node (v5) {};
%\path (130:0.5cm) node (v6) {};
%\path (170:0.5cm) node (v7) {};
%\path (250:0.5cm) node (v8) {};
%\path (290:0.5cm) node (v9) {};
%
%\path (0:0.75cm) node (v10) {};
%\path (20:0.75cm) node (v11) {};
%
%
%\path (40:0.75cm) node (v12) {};
%\path (60:0.75cm) node (v13) {};
%
%\path (120:0.75cm) node (v14) {};
%\path (140:0.75cm) node (v15) {};
%
%\path (160:0.75cm) node (v16) {};
%\path (180:0.75cm) node (v17) {};
%
%\path (240:0.75cm) node (v18) {};
%\path (260:0.75cm) node (v19) {};
%
%\path (280:0.75cm) node (v20) {};
%\path (300:0.75cm) node (v21) {};

%\draw (v0) -- (v1)
%(v0) -- (v2)
%(v0) -- (v3)
%(v1) -- (v4)
%(v1) -- (v5)
%(v2) -- (v6)
%(v2) -- (v7)
%(v3) -- (v8) 
%(v3) -- (v9)
%(v4) -- (v10)
%(v4) -- (v11)
%(v5) -- (v12)
%(v5) -- (v13)
%(v6) -- (v14) 
%(v6) -- (v15)
%(v7) -- (v16)
%(v7) -- (v17)
%(v8) -- (v18)
%(v8) -- (v19)
%(v9) -- (v20)
%(v9) -- (v21);

%\node (LT) at (-7,1) {Log Terminal};
%\node (LC) at (-7,-1){Log Canonical};
%\node (R) at (-3,1){Rational};
%\node(DB) at (-3,-1){Du Bois};
% 
%
%\draw[->] (LT) edge[double] (LC);
%\draw[->] (LT) edge[double] (R);
%\draw[->] (R) edge[double] (DB);
%\draw[->] (LC) edge[double] (DB);

\end{tikzpicture}
\caption{Vectors which vanish on the colored nodes comprise the kernel of $\W_{\ZZ_2^2}(k) \onto \W_{\ZZ_2}(k)$ }
\label{fig:Primeideal}
\end{figure}
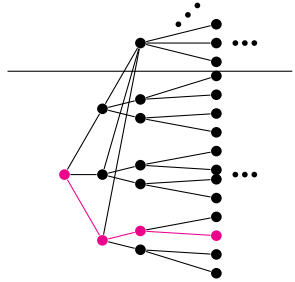

The ring homomorphism $\W_{\ZZ_p^d}(k) \onto \W_{\ZZ_p}(k)$ induced by $f$ is projection onto the coordinates indexed by (\ref{eq:prime}), so the kernel $\fp_f$ of this ring homomorphism is a prime of coheight one in $\W_{\ZZ_p^d}(k)$. All cyclic $\ZZ_p^d$-sets have the form $\sum_{i=1}^{d-1} \ZZ_p e_i+ \ZZ_p p^r e_d$ for some $\ZZ_p$-basis $\{e_1,\ldots,e_d\}$ and $r \geq 0$, so $$\bigcap_{f : \ZZ_p^d \onto \ZZ_p} \fp_f = \{ \mbf{a} \in \W_{\ZZ_p^d}(k) \colon a_T = 0 \mbox{ if } T \mbox{ is cyclic}\},$$ which is not $\{\mbf{0}\}$ since cyclic $\ZZ_p^d$-sets have level $0$; see Figure~\ref{fig:Z2Prime1}. The nilradical of $\W_{\ZZ_p^2}(k)$ is $\{\mbf{0}\}$ by \cite[Thm. 5. 17]{Mil}, therefore there must be additional non-maximal prime ideals in $\W_{\ZZ_p^2}(k)$ beyond those we have just constructed. The search for these extra prime ideals motivated Theorem~\ref{thm:diminfinite} which calculates the dimension of $\W_{\ZZ_p^d}(k)$ with $d \geq 2$. Also, when $d = 2$, all $\ZZ_p^2$-sets of level $0$ are cyclic and the $\ZZ_p^2$-sets of level $0$ (shaded portion of Figure~\ref{fig:Z2Prime1}) are in bijective correspondence with the tree of $\ZZ_p$-lattices in $\QQ_p^2$ up to scaling, gave us the perspective to prove Theorem~\ref{thm:isomorphism}. 

\begin{figure}[htb]
\begin{center}
\scalebox{0.6}{
\begin{tikzpicture}[smooth]

% Draw axes

\draw[-] (0,0) -- (1,0);
\draw[-] (0,0) -- (1,-1);
\draw[-] (0,0) -- (1,1);

\draw[-] (1,1) -- (3,1.25);
\draw[-] (1,1) -- (3,0.75);

\draw[-] (1,0) -- (3,0.25);
\draw[-] (1,0) -- (3,-0.25);

\draw[-] (1,-1) -- (3,-0.75);
\draw[-] (1,-1) -- (3,-1.25);

\draw[-] (-1,1.5) -- (6,1.5);

\draw[-] (1,1) -- (3,2.5);
\draw[-] (1,0) -- (3,2.5);
\draw[-] (1,-1) -- (3,2.5);

\draw[-] (3,2.5) -- (5,3);
\draw[-] (3,2.5) -- (5,2.5);
\draw[-] (3,2.5) -- (5,2);

\draw[-] (3,1.25) -- (5,1.35);
\draw[-] (3,1.25) -- (5,1.1);

\draw[-] (3,0.75) -- (5,0.85);
\draw[-] (3,0.75) -- (5,0.6);

\draw[-] (3,0.25) -- (5,0.35);
\draw[-] (3,0.25) -- (5,0.1);

\draw[-] (3,-0.25) -- (5,-0.15);
\draw[-] (3,-0.25) -- (5,-0.4);

\draw[-] (3,-0.75) -- (5,-0.65);
\draw[-] (3,-0.75) -- (5,-0.9);

\draw[-] (3,-1.25) -- (5,-1.2);
\draw[-] (3,-1.25) -- (5,-1.45);

\draw[-] (3,1.25) -- (5,3);
\draw[-] (3,0.75) -- (5,3);

\draw[-] (3,0.25) -- (5,2.5);
\draw[-] (3,-0.25) -- (5,2.5);

\draw[-] (3,-0.75) -- (5,2);
\draw[-] (3,-1.25) -- (5,2);

\fill[blue] (0,0) circle (0.1cm);

\fill[blue] (1,0) circle (0.1cm);
\fill[blue] (1,-1) circle (0.1cm);
\fill[blue] (1,1) circle (0.1cm);

\fill[blue] (3,1.25) circle (0.1cm);
\fill[blue] (3,0.75) circle (0.1cm);

\fill[blue] (3,0.25) circle (0.1cm);
\fill[blue] (3,-0.25) circle (0.1cm);

\fill[blue] (3,-0.75) circle (0.1cm);
\fill[blue] (3,-1.25) circle (0.1cm);

\fill[black] (3,2.5) circle (0.1cm);

\fill[black] (5,3) circle (0.1cm);
\fill[black] (5,2.5) circle (0.1cm);
\fill[black] (5,2) circle (0.1cm);

\fill[blue] (5,-1.2) circle (0.1cm);
\fill[blue] (5,-1.45) circle (0.1cm);

\fill[blue] (5,-0.65) circle (0.1cm);
\fill[blue] (5,-0.9) circle (0.1cm);

\fill[blue] (5,-0.15) circle (0.1cm);
\fill[blue] (5,-0.4) circle (0.1cm);

\fill[blue] (5,0.35) circle (0.1cm);
\fill[blue] (5,0.1) circle (0.1cm);

\fill[blue] (5,0.85) circle (0.1cm);
\fill[blue] (5,0.6) circle (0.1cm);

\fill[blue] (5,1.35) circle (0.1cm);
\fill[blue] (5,1.1) circle (0.1cm);

\fill[black] (7,0) circle (0.05cm);
\fill[black] (7.2,0) circle (0.05cm);
\fill[black] (7.4,0) circle (0.05cm);

\fill[black] (5.5,2.5) circle (0.05cm);
\fill[black] (5.7,2.5) circle (0.05cm);
\fill[black] (5.9,2.5) circle (0.05cm);

\fill[black] (4,3.25) circle (0.05cm);
\fill[black] (4.2,3.45) circle (0.05cm);
\fill[black] (4.4,3.65) circle (0.05cm);
%\fill[black] (5.7,2.5) circle (0.05cm);
%\fill[black] (5.9,2.5) circle (0.05cm);

\node at (-0.75,0) {$\ZZ_2^2/\ZZ_2^2$} {};
\node at (2,2.5) {$\ZZ_2^2/2 \ZZ_2^2$} {};

%\draw (5,3) circle (0.15cm);
%\draw (5,2.5) circle (0.15cm);
%
%\draw (5,-1.2) circle (0.15cm);
%\draw (5,-1.45) circle (0.15cm);
%
%\node at (5.4,-1.15) {$T_1$} {};
%\node at (5.4,-1.5) {$T_2$} {};
%
%\node at (5.4,3) {$V_1$} {};
%\node at (5.4,2.2) {$V_2$} {};

\end{tikzpicture}
}
\end{center}
\caption{Intersection of prime ideals for $G = \ZZ_2^2$}
\label{fig:Z2Prime1}
\end{figure}
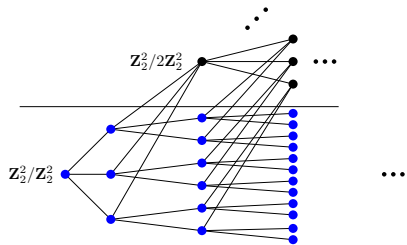

\section{Topology of Witt-Burnside rings}\label{sec:Top}

%We can consider any $G/N$-set $T$ as a $G$-set by defining the action of $g \in G$ on an element $x \in T$ to be $g \cdot x := gN \cdot x$.   A $G$-set arises in this way from a $G/N$-set precisely when 
%$N$ acts trivially on it. A $G/N$-set is transitive as a $G/N$-set if and only if it's transitive as a $G$-set since its $G/N$-orbits and $G$-orbits are the same.  Writing a transitive $G$-set as $G/H$, $N$ acts trivially on it if and only if $N \subset H$, in which 
%case we can write $G/H$ as $(G/N)/(H/N)$ to see it as a $G/N$-set. Two $G/N$-sets are isomorphic if and only if they are isomorphic when viewed as $G$-sets.  So turning $G/N$-sets into $G$-sets 
%gives us an embedding of $\mcal{F}(G/N)$ into $\mcal{F}(G)$ and the image of $\mcal{F}(G/N)$ in $\mcal{F}(G)$ is all $T \in \mcal{F}(G)$ on which $N$ acts trivially.

%For $T \in \mcal{F}(G)$ which is in the image of $\mcal{F}(G/N)$, write $T \cong G/H$ as a $G$-set with $N \subset H$.
%If $T' \leq T$ then we can write $T' \cong G/H'$ where $H \subset H'$, so $N \subset H'$. 
%Therefore the embedding of $\mcal{F}(G/N)$ into $\mcal{F}(G)$ is ``full'' as a partially ordered subset in the sense that
%all finite transitive $G$-sets below a $G$-set in the image are also in the image. 

In this section we describe three topologies on the Witt-Burnside ring $\W_G(A)$. {\it Unless otherwise stated, for this section, fix $A$ to be a commutative ring and $G$ is a profinite group. }

For any {\it{closed}} normal subgroup $N$ of $G$ the quotient $G/N$ is a profinite group and 
there is a natural group homomorphism from $G$ onto $G/N$. Identifying $\mcal{F}(G/N)$ with its image in $\mcal{F}(G)$, there is a natural projection map 
\begin{eqnarray*}
\tn{Proj}_{G/N}^G : \mbf{W}_G(A) &  \to & \mbf{W}_{G/N}(A) \\
\mbf{a} & \mapsto & 
(a_T)_{T \in \mcal{F}(G/N)}.
\end{eqnarray*} 
This is trivially surjective and is a ring homomorphism. 

\begin{example}
The ring homomorphism $\W_{\ZZ_p}(A) \rightarrow \W_{\ZZ_p/p^n\ZZ_p}(A)$ 
is the classical truncation homomorphism 
from $\W(A)$ onto the length $n+1$ Witt vectors $\W_{n+1}(A)$.
\end{example}

%If $G$ and $H$ are profinite groups and there is a continuous surjective group homomorphism $G \twoheadrightarrow H$ then there is a surjective ring homomorphism $\W_G(A) \twoheadrightarrow \W_H(A)$ since we can realize $H$ as some $G/N$ and the map $G \twoheadrightarrow H$ as reduction mod $N$. 

\begin{example}
For any $d \geq 2$ and field $k$ of characteristic $p$, $\W_{\ZZ_p}(k)$ is a homomorphic image of $\W_{\ZZ_p^d}(k)$ since $\ZZ_p$ is a continuous homomorphic image of $\ZZ_p^d$ (in many ways). Since $\W_{\ZZ_p}(k)$ is a domain, we obtain as kernels many non-maximal prime ideals of $\W_{\ZZ_p^d}(k)$. Unfortunately, this does not give all the not give us all the non-maximal prime ideals as discussed in Subsection~\ref{sec:PI}.
\end{example}

\begin{example}\label{homZpr}
For any $d \geq 2$, the ring $\W_{\ZZ_p^d}(A)$ has $\W_{\ZZ_p^2}(A)$ as a homomorphic image, since 
the group $\ZZ_p^d$ has $\ZZ_p^2$ as a continuous homomorphic image. 
\end{example}

\begin{definition}
\label{Kdef1}
For any closed normal subgroup $N \lhd G$, set 
\begin{eqnarray*}
K_N = \ker(\tn{Proj}_{G/N}^G) & = & \{\mbf{a} \in \W_G(A) : a_T = 0 \text{ for all } T \in \mcal{F}(G/N) \} \\
                              & = & \{\mbf{a} \in \W_G(A) : a_T = 0 \text{ when } N \text{ acts trivially on } T \}.
\end{eqnarray*}
\end{definition}

The kernel $K_N$ is an ideal of $\W_G(A)$ and the rings $\W_G(A)/K_N$ and $\W_{G/N}(A)$ are isomorphic. 
\begin{example}
Taking $N = G$, 
$K_G = \{\mbf{a} \in \W_G(A) : a_0 = 0\}$ and the quotient $\W_G(A)/K_G \cong W_0(A) = A$. 
If $A$ is a domain (field) then $K_G$ is a prime (maximal) ideal in $\W_G(A)$. 
\end{example}
%(To be more explicit in its dependence on the group and the ring, we could write
%$K_N$ as $K_N(G,A)$, but we will leave this dependence out of the notation unless it is 
%important to be explicit about it.) 
%If $N \subset N'$ and $a_T = 0$ whenever $N$ acts trivially on $T$ then 
%$a_T = 0$ whenever $N'$ acts trivially on $T$, so $K_N \subset K_{N'}$.   
For open normal subgroups $N$ and $N'$, 
the groups $G/N$ and $G/N'$ are finite. 
The rings $\mbf{W}_{G/N}(A)$ for open normal $N$ with the ring homomorphisms 
$$
\tn{Proj}_{G/N'}^{G/N} : \mbf{W}_{G/N}(A)  \to  \mbf{W}_{G/N'}(A) 
$$
when $N \subset N'$ 
form a projective system and it is isomorphic to the projective system of rings 
$\W_G(A)/K_N$ and natural ring homomorphisms 
$\W_G(A)/K_N \to  \W_G(A)/K_{N'}$ when $N \subset N'$. 

We have 
$\mbf{W}_G(A) \cong \varprojlim\limits_{N} \mbf{W}_{G/N}(A) \cong \varprojlim\limits_{N} \W_G(A)/K_N$ as rings, where the inverse limits are taken over open normal subgroups of $G$ ordered by reverse inclusion. This gives $\W_G(A)$ a topology as a closed subset of the product $\prod_{N} \W_{G/N}(A) \cong \prod_{N} \W_G(A)/K_N$, where $N$ runs over the open normal subgroups of $G$ and each factor $\W_{G/N}(A)$ and $\W_G(A)/K_N$ is given the discrete topology.   This topology on $\W_G(A)$ is called the {\it profinite topology}. The profinite topology on $\W_G(A)$ is Hausdorff since any subspace of a product of discrete spaces is Hausdorff. The ring $\W_G(A)$ is compact in its profinite topology if and only if $A$ is finite.  We note that $\W_G(A)$ is complete in the profinite topology which follows from the proof of \cite[Thm 3.3.2]{DS88}, see also \cite[pg. 357]{Ell06}.

%We remark that $\W_G(A) = \varprojlim_{N} \W_{G/N}(A)$ where the inverse limit runs over open normal subgroups of $G$ and, that giving each $\W_{G/N}(A)$ the discrete topology induces a natural profinite topology on $\W_G(A)$ in which $\W_G(A)$ is complete, which follows from the proof of \cite[Thm 3.3.2]{DS88}, see also \cite[pg. 357]{Ell06}. 

%%MAYBE CUT THIS PARAGRAPH
To say $\mbf{a}$ and $\mbf{b}$ are near each other in the profinite topology means $\mbf{a} \equiv \mbf{b} \bmod K_N$ for some ``small''
open normal subgroup $N$ in $G$; the larger $[G:N]$ is, the closer $\mbf{a}$ and $\mbf{b}$ are.
How can we tell in terms of coordinates of $\mbf{a}$ and $\mbf{b}$ that 
$\mbf{a} \equiv \mbf{b} \bmod K_N$? 
While $K_N$ is described by Definition \ref{Kdef1} for all closed $N$, when $N$ is open the set 
$G/N$ belongs to $\mcal{F}(G)$ and $\mcal{F}(G/N)$ is a ``full'' subset of $\mcal{F}(G)$, so we can describe $K_N$ a little differently than in Definition \ref{Kdef1}:
\begin{equation}\label{Kdef2}
K_N = \{\mbf{a} \in \W_G(A) : a_T = 0 \text{ for all } T \in \mcal{F}(G) \mbox{ such that } T \leq G/N\}. 
\end{equation}

\begin{lemma}
\label{lem:closetop}
For any open normal subgroup $N$of $G$, 
$$
\mbf{a} \equiv \mbf{b} \bmod K_N \Longleftrightarrow a_T = b_T \text{ for all } T \leq G/N.
$$
\end{lemma}
\begin{proof}
Assume $a_T = b_T$ for $T \leq G/N$. Using Theorem~\ref{thmRS} with the partition 
$$\{T \leq G/N\} \cup (\mcal{F}(G) - \{T \leq G/N\})$$ 
we can write $\mbf{a} = \mbf{a}_1 + \mbf{a}_2$ where 
$\mbf{a}_1$ is supported on $T \leq G/N$ and $\mbf{a}_2$ is supported on 
the complement, so $\mbf{a}_2 \in K_N$. 
Similarly write $\mbf{b} = \mbf{b}_1 + \mbf{b}_2$. Then $\mbf{a}_1 = \mbf{b}_1$, so 
$\mbf{a}-\mbf{b} = \mbf{a}_2 - \mbf{b}_2 \in K_N$. 

Conversely, assume $\mbf{a} \equiv \mbf{b} \bmod K_N$ and write $\mbf{a} = \mbf{b} + \mbf{c}$ for $\mbf{c} \in K_N$. Since $S_0(\underline{X},\underline{Y}) = X_0 + Y_0$, we have 
$a_0 = b_0 + c_0$.  Since $c_0 = 0$, $a_0 = b_0$. If $G/N$ is trivial we are done. Otherwise, let $T$ be a nontrivial $G$-set with $T \leq G/N$. 
Since the polynomial $S_T(\underline{X},\underline{Y})$ only depends on $Y_U$ for $U \leq T$,
and $c_U = 0$ for all $U \leq T$,  
evaluating $S_T$ with  $Y_U = c_U$ for all $U$ or $Y_U = 0$ for all $U$ 
leads to the same result. Therefore $a_T = S_T(\mbf{b},\mbf{c}) = S_T(\mbf{b},\mbf{0}) = b_T$,  where the last equality follows from Theorem \ref{thm:integralwitt}(b). 
\end{proof}

\begin{theorem}
The ring $\W_G(A)$ is complete in the profinite topology.
\end{theorem}
\begin{proof}
It suffices to show any Cauchy net converges. 
Let $\{\mbf{a}_i\}_{i \in I}$ be a Cauchy net in $\W_G(A)$ for the profinite topology.  For each open normal subgroup $N \lhd G$ there is an $i_N \in I$ such that for all $i,j \geq i_N$, 
$\mbf{a}_i \equiv \mbf{a}_j \bmod K_N(G,A)$, which by Lemma \ref{lem:closetop} means $(\mbf{a}_i)_T = (\mbf{a}_j)_T$ for all $T \leq G/N$. For each $T$, pick $N$ so that $T \leq G/N$ and let 
$a_T$ be the common value of $(\mbf{a}_i)_T$ for all $i \geq i_N$. 
This is independent of the choice of $N$. 
Let $\mbf{a}$ be the Witt vector with $T$-th coordinate $a_T$ for all $T$. 
We will show this is the limit of the $\mbf{a}_i$'s. For an open normal subgroup $N$ of $G$ and $i \geq i_N$, $a_T = (\mbf{a}_i)_T$ for all $T \leq G/N$. By Lemma \ref{lem:closetop}, we have $\mbf{a} \equiv \mbf{a}_i \bmod K_N(G,A)$ for $i \geq i_N$. Thus $\lim \mbf{a}_i = \mbf{a}$ in the profinite topology.
\end{proof}

\subsection{\Teich elements}

\begin{definition}\label{TTeich}
For $a \in A$ and $T \in \mcal{F}(G)$, denote by $\omega_T(a) \in \W_G(A)$ the Witt vector with $T$-coordinate $a$ and all other coordinates $0$. We call $\omega_T(a)$ the $T$-th {\it{Teichm\"uller lift}} of $a$. 
\end{definition}

Denote the trivial $G$-set $G/G$ as $0$ and by $\omega_0(a)$ the $G/G$-th Teichm\"uller lift. The function $\omega_0 \colon A \rightarrow \W_G(A)$ generalizes the classical Teichm\"uller lift. Like the classical Teichm\"uller lift, $\omega_0$ is multiplicative; see \cite[Thm. 2.8]{Mil} and for a general formula for $\omega_T(a)\omega_{T'}(b)$, see \cite[p. 355]{Ell06}. These are important as they form a topological basis (c.f., Theorem~\ref{Rpart}) and in Section~\ref{sec:vdP} we identify them as an analogue of a van der Put basis in certain function rings.

\begin{remark}
If $G$ is a pro-$p$ group and $k$ is a perfect field (or just a perfect ring) of characteristic 
$p$ then we can write any $a_U \in k$ in the form $b_U^{\#U}$ for some $b_U \in k$, so 
$$
\mbf{a} = \lim_{T \in \mcal{F}(G)} \sum_{U \leq T} \omega_U \left(b_U^{\#U} \right) = 
\lim_{T \in \mcal{F}(G)} \sum_{U \leq T} \omega_0(b_U)\omega_U(1)
$$
by \cite[Thm. 2.8]{Mil}.  This last series is the usual $p$-adic expansion of a 
Witt vector when $G = \ZZ_p$: $\omega_U(1) = p^n$ when $U = \ZZ_p/p^n\ZZ_p$ and $\lim_{T \in \mcal{F}(\ZZ_p)} = \lim_{n \to \infty} $. 
\end{remark}

Theorem \ref{thmRS}, which extends from a two-element partition of $\mathcal{F}(G)$ to any finite partition, also extends with the profinite topology on $\W_G(A)$ to suitable infinite partitions of $\mathcal{F}(G)$.

\begin{theorem}\label{Rpart}
Let $\mcal{F}(G) = \bigcup_{i \in I} R_i$ be a partition where the index set $I$ is a directed set such that for all $j \in I$, the set $\{ i \in I : i \leq j \}$ is finite.  For any 
$\mbf{a} \in \W_G(A)$, let $\mbf{r}_i(\mbf{a})$ be the Witt vector in 
$\W_G(A)$ which is the part of $\mbf{a}$ supported on $R_i$:
$$
\mbf{r}_i(\mbf{a}) = 
\begin{cases}
a_T, & \text{ if } T \in R_i, \\
0, & \text{ otherwise}.
\end{cases}
$$
For any $\mbf{a} \in \W_G(A)$,
$$
\mbf{a} = \lim_{j \in I} \sum_{i \leq j} \mbf{r}_i(\mbf{a}),
$$ where the limit is taken in the profinite topology on $\W_G(A)$. In particular, 
$$
\mbf{a} = \lim_{T \in \mcal{F}(G)} \sum_{U \leq T} \omega_U(a_U). 
$$
\end{theorem}
\begin{proof}
Consider the net $\{\mbf{s}_j\}_{j \in I}$ in $\W_G(A)$ defined by $\mbf{s}_j = \sum_{i \leq j} \mbf{r}_i(\mbf{a})$. Notice for each $j$ the sum $\sum_{i \leq j} \mbf{r}_i(\mbf{a})$ has finitely many terms by hypothesis. We aim to show $\{\mbf{s}_j\}_{j \in I}$ converges to $\mbf{a}$ in the profinite topology, that is to say for each open normal subgroup $N$ of $G$ there is $j_N \in I$ such that for all $j \geq j_N$, $\mbf{s}_j \equiv \mbf{a} \bmod K_N$. By Lemma \ref{lem:closetop} this means 
the $T$-th components of $\mbf{s}_j$ and $\mbf{a}$ agree for all $j \geq j_N$ 
and $T \leq G/N$.

For each $T \leq G/N$, there is a unique index $i_T \in I$ such that $T \in R_{i_T}$. Therefore $\mbf{r}_{i_T}(\mbf{a})_T = a_T$. Let $S_N = \{ i_T : T \leq G/N \}$. This is a finite set of indices in $I$. If $i \not\in S_N$ then $T \not\in R_i$ for all $T \leq G/N$, so $\mbf{r}_i(\mbf{a})_T = 0$ when $T \leq G/N$. Letting $R = I \setminus S_N$, $\mbf{a} = \sum_{i \in S_N} \mbf{r}_i(\mbf{a}) + \mbf{r}_R(\mbf{a})$ by an extension of Theorem \ref{thmRS} to finite partitions of $\mcal{F}(G)$. If $T \leq G/N$ then $i_T \notin R$, so $\mbf{r}_R(\mbf{a})_T = 0$. Therefore $\mbf{r}_R(\mbf{a}) \equiv 0 \bmod K_N$ so 

\begin{equation}
\label{eq:converge}
\mbf{a} \equiv \sum_{i \in S_N} \mbf{r}_i(\mbf{a}) \bmod K_N.
\end{equation}

Since $I$ is directed and $S_N$ is finite, there is an index $j_N \in I$ such that $j_N \geq i$ for all $i \in S_N$. For all $j \geq j_N$, $S_N \subset \{ i \in I : i \leq j \},$ so $$\mbf{s}_j = \sum_{i \leq j} \mbf{r}_i(\mbf{a}) = \sum_{i \in S_N} \mbf{r}_i(\mbf{a}) + \sum_{\stackrel{i \leq j}{i \notin S_N}} \mbf{r}_i(\mbf{a}).$$ Reducing mod $K_N$, the first summand is congruent to $\mbf{a} \bmod K_N$ by (\ref{eq:converge}) whereas the second summand is congruent to $0 \bmod K_N$. The last conclusion is clear using the partition of $\mcal{F}(G)$ into singleton subsets. 
\end{proof}

\subsection{Initial Vanishing Topology}\label{sec:CT}

In the profinite topology on $\W_G(A)$, an element 
$\mbf{a}$ is small if $a_T = 0$ for all $T$ on which 
a small open normal subgroup of $G$ (one of large index in $G$) acts trivially. 
There is another notion of being small: $\mbf{a}$ is small when $a_T = 0$ for 
all $T$ of size below a given bound. This topology is defined by another family of ideals.

\begin{definition}\label{Indef} $($c.f., \cite[Def. 2.12]{Mil}$)$ 
For $n \in \mbf{Z}^+$, set $$I_n(G,A) = \{ \mbf{a} \in \W_G(A) : a_T = 0 \mbox{ for }  \# T < n \}.$$
\end{definition}

%\begin{lemma}
%\label{lem:fil1}
%Each $I_n(G,A)$ is an ideal in $\mbf{W}_G(A)$. 
%\end{lemma}
%\begin{proof}
%Each addition polynomial $S_T$ depends only on variables indexed by (isomorphism classes of) finite transitive $G$-sets $U \leq T$ and 
%has no constant term by Theorem \ref{thm:integralwitt}. So if two Witt vectors are in $I_n(G,A)$ then their sum is also in $I_n(G,A)$. 
%
%It remains to show for any $\mbf{a} \in \W_G(A)$ and $\mbf{b} \in I_n(G,A)$ that $\mbf{a}\mbf{b} \in I_n(G,A)$. 
%Set $\mbf{c} = \mbf{a}\mbf{b}$.
%By the definition of multiplication in $\W_G(A)$, $c_T = M_T(\mbf{a},\mbf{b})$ for any $T$. 
%If $\# T < n$, $b_U = 0$ for $U \leq T$ by hypothesis.  Since 
%$M_T(\underline{X},\underline{Y})$ only depends on $Y_U$ for $U \leq T$, 
%$c_T = M_T(\mbf{a},\mbf{b}) = M_T(\mbf{a},\mbf{0})$ and 
%$M_T(\mbf{a},\mbf{0}) = 0$ by 
%Theorem \ref{thm:integralwitt}.
%\end{proof}

The sets $I_n(G,A)$ are in fact ideals \cite[Lem 2.13]{Mil}. These are the Witt vectors $\mbf{a}$ with support in $\{ T : \# T \geq n \}$. Obviously the ideals $I_n(G,A)$ are decreasing:
\begin{equation}\label{eq:filtration}
\W_G(A) = I_1(G,A) \supset I_2(G,A) \supset \cdots \supset I_n(G,A) \supset I_{n+1}(G,A) \supset \cdots.
\end{equation}
Since $\bigcap_{n \geq 1} I_n(G,A) = \{\mbf{0}\}$, we can put a Hausdorff topology on 
$\W_G(A)$ so that the collection of ideals $\{I_n(G,A)\}_{n \geq 1}$ form a fundamental system of neighborhoods of $\mbf{0}$: 
a nonempty subset $\mcal{O} \subset \W_G(A)$ is open in this topology if for each 
$\mbf{a} \in \mcal{O}$ the coset $\mbf{a} + I_n(G,A)$ is in $\mcal{O}$ for some $n$.
This is called the {\it initial vanishing topology}.

\begin{remark}
Of course it could happen that $I_n(G,A) = I_{n+1}(G,A)$, namely if there are no open subgroups of $G$ with index $n$. If $G$ is an infinite pro-$p$ group, then the different ideals in this filtration are indexed by $p$-powers, and ideals indexed by different powers of $p$ 
are different from each other since there are open subgroups of each $p$-power index in $G$: 
$$
I_{p^m}(G,A) = \{\mbf{a} : a_T = 0 \text{ for } \# T < p^{m}\} %= \{\mbf{a} : a_T = 0 \text{ for } \# T \leq p^{m-1}\}
$$ 
and 
$$
\W_G(A) = I_1(G,A) \supsetneqq I_p(G,A) \supsetneqq \cdots \supsetneqq I_{p^m}(G,A) \supsetneqq I_{p^{m+1}}(G,A) \supsetneqq \cdots.
$$
\end{remark}

We record a lemma which is analogue of Lemma \ref{lem:closetop} whose proof is similar and omitted before proceeding to prove that $\W_G(A)$ is also complete in the initial vanishing topology. 

\begin{lemma}
\label{lem:Itopcomplete}
For $n \geq 1$ and $\mbf{a}$ and $\mbf{b}$ in $\W_G(A)$, $\mbf{a} \equiv \mbf{b} \bmod I_n(G,A)$ if and only if $a_T = b_T$ for all $T \in \mcal{F}(G)$ such that $\# T < n$. 
\end{lemma}

\begin{theorem}
\label{thm:Itopcomplete}
The ring $\W_G(A)$ is complete in the initial vanishing topology.
\end{theorem}

\begin{proof}
Since the topology is defined by a countable set of ideals, 
it suffices to show any Cauchy sequence (rather than Cauchy net) converges. 
Let $\{\mbf{a}_i\}_{i \in \mbf{N}}$ be a Cauchy sequence in $\W_G(A)$ for the initial vanishing 
topology.  We will construct a candidate limit and then show it works. 

For each $n \geq 1$ there is an integer $N_n \in \ZZ^+$ such that for all $r,s \geq N_n$, 
$\mbf{a}_r \equiv \mbf{a}_s \bmod I_n(G,A)$, which by Lemma \ref{lem:Itopcomplete} means $(\mbf{a}_r)_T = (\mbf{a}_s)_T$ for all $T \in \mcal{F}(G)$ such that $\# T < n$. 
We may choose the numbers $N_n$ so that $N_1 < N_2 < N_3 < \cdots$. 

For each $T$, pick $n > \#T$ and let 
$a_T$ be the common value of $(\mbf{a}_r)_T$ for all $r \geq N_n$. 
This is independent of the choice of $n$. 
Let $\mbf{a}$ be the Witt vector with $T$-th coordinate $a_T$ for all $T$. 
We will show this is the limit of the $\mbf{a}_r$'s. 

For each $n \geq 1$ and $T \in \mcal{F}(G)$ with $\# T < n$, 
$a_T = (\mbf{a}_r)_T$ for $r \geq N_n$, so by Lemma \ref{lem:Itopcomplete} $\mbf{a} \equiv \mbf{a}_r \bmod I_n(G,A)$ for $r \geq N_n$. Thus $\lim_{r \to \infty} \mbf{a}_r = \mbf{a}$ in the 
initial vanishing topology.
\end{proof}

Set $I_n = I_n(G,A)$. 
How do the profinite and initial vanishing topologies, defined by the families of ideals $\{K_N\}_{N \text{ open}}$ and $\{I_n\}_{n \geq 1}$, compare with each other?

\begin{lemma}
\label{lem:Poincare}
Given finitely many finite transitive $G$-sets $T_1,\dots,T_m$,
there is an open normal subgroup $N$ in $G$ such that $G/N \geq T_i$ for $1 \leq i \leq m$.
\end{lemma}

\begin{proof}
The frame $\mcal{F}(G)$ is a directed set, so there is a $T \in \mcal{F}(G)$ with $T \geq T_i$ for $i = 1,\ldots,m$. Since the $G$-sets with normal stabilizers are cofinal in $\mcal{F}(G)$, we can take $T = G/N$ with $N$ normal. 
\end{proof}

\begin{theorem}
\label{thm:sametop}
Every $K_N$, for open normal $N$, is open in the initial vanishing 
topology on $\W_G(A)$.  If $G$ is topologically finitely generated then 
the profinite topology and the initial vanishing topology on $\W_G(A)$ are the same.
\end{theorem}

\begin{proof}
Let $N$ be an open normal subgroup of $G$.
Since $K_N$ is an additive group, to show it is 
open in the initial vanishing topology it suffices (and in fact is equivalent) to show there is an 
$n \geq 1$ such that 
$I_n \subset K_N$.  
Let $n$ be any integer exceeding $[G:N]$.  Then 
$T \leq G/N \Rightarrow \#T < n$, so 
for all $\mbf{a} \in \W_G(A)$ we have 
$$
a_T = 0 \text{ for } \# T < n \Longrightarrow a_T = 0 \text{ for } T \leq G/N. 
$$
Thus $I_n \subset K_N$.

Now assume $G$ is topologically finitely generated.
Since $I_n$ is an additive group, to show $I_n$ is open in the profinite topology on $\W_G(A)$ 
it suffices to find an open normal subgroup $N \lhd G$ such that 
$K_N \subset I_n$.  The key fact we use about a topologically finitely generated profinite group is that it 
has finitely many open subgroups of each index \cite[p. 45]{RZ00}, and thus 
$\{T : \#T < n\}$ is finite for all $n$.  Therefore Lemma \ref{lem:Poincare} tells us 
there is an open normal subgroup $N \lhd G$ such that
$\#T < n \Rightarrow T \leq G/N$, so $K_N \subset I_n$. 
\end{proof}

When $G$ is topologically finitely generated,
$\{T : \#T = n\}$ is finite for every $n$, so Theorem \ref{Rpart} implies 
$$
\mbf{a} =  \sum_{n \geq 1} \sum_{\#T = n} \omega_T(a_T) 
$$
for all $\mbf{a}$ in $\W_G(A)$, where we use the partition $\mcal{F}(G) = \bigcup_{n \geq 1} R_n$ for $R_n = \{ T : \# T = n\}$.  

\begin{remark}
For open normal subgroup $N$ in $G$, to say $K_N \subset I_n$ is the same as saying $\#T < n \Rightarrow T \leq G/N$.
Since $[G:N] < \infty$, there are only finitely many open subgroups of $G$ 
containing $N$, so if $K_N \subset I_n$ for some $N$ there can be only finitely many $T$ in $\mcal{F}(G)$  
with size less than $n$.  Considering this over all $n$,  
the profinite topology and the initial vanishing topology on $\W_G(A)$
are equal if and only if $G$ has finitely many conjugacy classes of open subgroups of 
each index.  
\end{remark}

It is natural to ask if the filtration \eqref{eq:filtration} is graded, i.e., $I_mI_n \subset I_{mn}$ for all $n, m \geq 1$. A partial result was shown when $G$ is any pro-$p$ group \cite[Lem 2.14]{Mil}. For pro-$p$ groups satisfying a mild technical condition (which is always true when the group is abelian) we check that the containment $I_{p^m}I_{p^n} \subset I_{p^{m+n}}$ holds in cases that are likely to be of most interest and in other cases it need not hold. 

\begin{definition}
\label{dff:ratio}
A pro-$p$ group $G$ is said to have the {\it ratio property} when 
the following holds: for all finite transitive $G$-sets $T$, $T_1$, and $T_2$ such that 
$T \geq T_1$, $T \geq T_2$, and $\#T < \#T_1\#T_2$, 
the ratio $\varphi_T(T_1)\varphi_T(T_2)/\varphi_T(T)$ is an integral multiple of $p$.
\end{definition}

If $T$ is a normal $G$-set, $\varphi_T(U) = \#U$ for any $U \leq T$, so the ratio 
$\varphi_T(T_1)\varphi_T(T_2)/\varphi_T(T)$ is $\#T_1\#T_2/\#T$, which  
is a power of $p$ greater than 1.  Therefore the ratio property 
is true when $G$ is a pro-$p$ abelian group. 

\begin{remark}Even though this property was only defined for pro-$p$ groups, one can define a more general ratio property for profinite groups by asking for all finite transitive $G$-sets $T$, $T_1$, and $T_2$ such that $T \geq T_1$ and $T \geq T_2$ that $\# T$ is a proper factor of $\# T_1 \# T_2$. It can be checked explicitly that it need not hold in general, for example it fails for symmetric groups $S_n$ with $n \geq 4$. The authors do not know any pro-$p$ groups for which the ratio property fails. 
%. The generalized ratio property holds for all dihedral groups of order up to $1200$, however it fails for the symmetric group $S_n$ for $n = 4, \ldots, 8$, for the alternating groups $A_n$ for $n = 5, \ldots, 9$, the groups $\GL_m(\ZZ/n\ZZ)$ for $m = 2$ and $n = 3,4,5,6,7$ and $m = 3$ and $n = 2,3$ and $m = 4$ and $n = 2$, as well as $\SL_m(\ZZ/n\ZZ)$ for $m = 2$ and $n = 3,4,5,6$ and $m = 3$ and $n = 2,3$. The authors do not know any pro-$p$ groups for which the ratio property fails. 
\end{remark}

The ratio property was found because of its role in the proof of Theorem \ref{thm:Fill} below.   
Passing from $G$ to $G/N$ where $N$ acts trivially on $T$, the ratio 
property for $T,T_1$ and $T_2$ can be viewed in $\mcal{F}(G/N)$ 
so it suffices to check Definition \ref{dff:ratio} for all finite
quotients of $G$. Therefore, the ratio property is really a hypothesis about finite $p$-groups. 

\begin{lemma}
\label{lem:Fill}
Let $G$ be a pro-$p$ group that satisfies the ratio property and $A$ be a ring of characteristic $p$. Let $V \in \mcal{F}(G)$ with $\# V \leq p^{m+n}$ for positive integers $m$ and $n$. Consider Witt vectors $\mbf{a}$ and $\mbf{b}$ in $\W_G(A)$ such that $a_U = 0$ for all $U < V$ such that $\# U < p^m$ and $b_U = 0$ for all $U < V$ such that $\# U < p^n$. Set $\mbf{c} = \mbf{ab}$. Then $c_V = 0$. 
\end{lemma}

\begin{proof}
This will follow from functoriality by proving the following mod $p$ congruence for particular Witt vectors over the ring 
$R = \mbf{Z}[\underline{X},\underline{Y}]$. 
Define $\mbf{x},\mbf{y} \in \W_G(R)$ by 
\begin{displaymath}
x_T = 
\begin{cases}
0, & \text{$\# T < p^m$ and $T < V$,} \\
X_T, & \textrm{otherwise,}
\end{cases}  \text{ and } 
y_T = 
\begin{cases}
0, & \text{$\# T < p^n$ and $T < V$,} \\
Y_T, & \textrm{otherwise.}
\end{cases}
\end{displaymath} 
Set $\mbf{z} = \mbf{x}\mbf{y}$.  We will show that 
\begin{equation}\label{Ucong}
T \leq V \Longrightarrow z_T \equiv 0 \bmod pR. 
\end{equation}

%Let's see how this implies the theorem.
%There is a ring homomorphism $f : R \to A$ such that $f(X_T) =  a_T$ and $f(Y_T) =  b_T$ for all $T \in \mcal{F}(G)$. 
%Then $\W_G(f)(\mbf{x}) = \mbf{a}$ and $\W_G(f)(\mbf{y}) = \mbf{b}$. 
%The product $\mbf{a}\mbf{b} = \W_G(f)(\mbf{z})$ has $T$-component 
%$f(z_T)$ and by (\ref{Ucong}) this lies in $f(pR) \subset pA = 0$ when $T \leq V$.

%Returning to the proof of (\ref{Ucong}), we argue by induction on $\#T$.
We argue by induction on $\#T$.
If  $\#T = 1$ then $T = 0$ and $z_0 = x_0y_0 = 0$. 
Let $p^r \leq p^{m+n}$ with $r \geq 1$ and 
assume by induction that for all $U < V$ such that $\# U < p^r$, $z_U \equiv 0 \bmod pR$. 
Pick $T \leq V$ with $\# T = p^r$.
Since the $T$-th Witt polynomial is a multiplicative function  
$W_T \colon \W_G(R) \rightarrow R$, 
$W_T(\mbf{z}) = W_T(\mbf{x})W_T(\mbf{y})$: 
$$\sum_{U \leq T} \varphi_T(U) z_U^{\# T / \#U } = \sum_{T_1,T_2 \leq T} \varphi_T(T_1)\varphi_T(T_2) x_{T_1}^{\#T / \# T_1} y_{T_2}^{\# T / \# T_2}.$$ Solving this equation for $z_T$ in $\QQ[\underline{X},\underline{Y}]$, 
\begin{equation}\label{ztf}
z_T = \sum_{T_1,T_2 \leq T} \frac{\varphi_T(T_1)\varphi_T(T_2)}{\varphi_T(T)} x_{T_1}^{\#T / \# T_1} y_{T_2}^{\# T / \# T_2} - \sum\limits_{U < T} \frac{\varphi_T(U)}{\varphi_T(T)} z_U^{\#T / \#U}.
\end{equation}

Since $z_U \in pR$ for $U < T$,
the second term in (\ref{ztf}) is $0$ mod $pR$ by \cite[Lem. 2.11]{Mil}. %Lemma \ref{lem:cong}.
In the first term  in (\ref{ztf}), 
if $T_2 \neq 0$ and $T_1 = V$  then $T_1 = T = V$ and $\varphi_T(T_1)\varphi_T(T_2)/\varphi_T(T) \equiv \varphi_T(T_2) \equiv 0 \bmod pR$ by \cite[Lem 2.10]{Mil}. %Lemma \ref{lem:propdivis} 
If $T_2 = 0$ then $y_{T_2} = 0$. A similar argument holds if $T_2 = V$ so we can assume $T_1 < V$ and $T_2 < V$. If either $\# T_1 < p^m$ or $\# T_2 < p^n$ then $x_{T_1} = 0$ or $y_{T_2} = 0$ respectively. The remaining terms in the first sum in (\ref{ztf}) have $T_1 < V$ with $\# T_1 \geq p^m$ and $T_2 < V$ with $\# T_2 \geq p^n$. In this case $\#T_1\#T_2 \geq p^{m+n} > p^r = \#T$, so by the ratio property  the coefficient in the first sum in (\ref{ztf}) is an integral multiple of $p$. Thus $z_T \equiv 0 \bmod pR$.     
\end{proof} 

\begin{theorem}
\label{thm:Fill}
Let $G$ be a pro-$p$ group.
If it satisfies the ratio property, 
then for all rings $A$ of characteristic $p$ and nonnegative integers $m$ and $n$, $$I_{p^m}(G,A) I_{p^n}(G,A) \subset I_{p^{m+n}}(G,A)$$ in $\W_G(A)$. In particular, this containment is true when $G$ is abelian. 
\end{theorem}
\begin{proof}
If either $m$ or $n$ is $0$, the result follows since $I_1(G,A) = \W_G(A)$. So we can assume $m > 0$ and $n > 0$. Let $\mbf{a} \in I_{p^m}(G,A)$ and $\mbf{b} \in I_{p^n}(G,A)$ and set $\mbf{c} = \mbf{ab}$. Let $V \in \mcal{F}(G)$ such that $\# V < p^{m+n}$. Since $a_U = 0$ for all $\# U < p^m$ and $b_U = 0$ for all $\# U < p^n$, by Lemma \ref{lem:Fill} $c_V = 0$, so $\mbf{c} \in I_{p^{m+n}}$. 
\end{proof}

\begin{remark}
Of course the hypothesis that $A$ has characteristic $p$ is necessary. In $\W_{\ZZ_p}(\ZZ)$, let $\mbf{x} = (0,1,0,0,\ldots) = \omega_{\ZZ_p/p\ZZ_p}(1)$. Clearly $\mbf{x} \in I_p(\ZZ_p,\ZZ)$ however one can quickly check that $\mbf{x}^2 = (0,p,\ldots)$, which is not in $I_{p^2}(\ZZ_p,\ZZ)$. 
\end{remark}

There is a surjective homomorphism $\tn{Proj}_{G/G}^G : \mbf{W}_G(k) \to \mbf{W}_0(k) \cong k$. The kernel $\fm = \{ \mbf{a} \in \mbf{W}_G(k) : a_0 = 0 \}$ is a maximal ideal in $\W_G(k)$. When $G$ is pro-$p$ and $k$ has characteristic $p$ this maximal ideal is unique \cite[Thm. 2.16]{Mil}. We also have $\fm = I_\ell$ where $\ell$ is the least integer greater than $1$ such that $G$ has an open subgroup of index $\ell$ (such an $\ell$ exists unless $G$ is the trivial group).  Therefore, when $G$ is a pro-$p$ group and $k$ a field of characteristic $p$, $\fm = I_p(G,k)$. In this setting we have three topologies: the $\fm$-adic topology, the initial vanishing (or $\{I_{p^n}\}$-adic) topology, and the profinite topology. We have already compared the last two topologies in Theorem \ref{thm:sametop}. 
As a consequence of Theorem~\ref{thm:Fill} or by \cite[Lem. 2.14]{Mil}, %By Corollary \ref{cor:IFillpower}, 
\begin{equation}
\label{eq:maxpowercontain}
\fm^n = I_p(G,k)^n \subset I_{p^n}(G,k),
\end{equation}
so $\bigcap_{n \geq 1} \fm^n = \{ \mbf{0} \}$, which shows the $\fm$-adic topology on $\W_G(k)$ is Hausdorff. 
Since $\fm^n \subset I_{p}^n \subset K_N$ for $N$ any open normal subgroup such that $[G:N] > p^n$, open sets in the profinite topology are also open in the initial vanishing topology, which are in turn open in the $\fm$-adic topology. 
It is not so clear how the higher powers of $\fm$ are related to the ideals $I_{p^n}$ other than by (\ref{eq:maxpowercontain}). 
In the case of $\W_{\ZZ_p}(k)$ where $k$ is a perfect field of characteristic $p$, i.e., the classical Witt vectors, the initial vanishing topology is the same as the $\fm$-adic topology since $\fm = (p) = I_p$ and $\fm^n = I_{p^n}$.

Utilizing tools from \cite{Mil} we show that the $\fm$-adic topology is different from the other two for $G = \ZZ_p^d$ with $d \geq 2$. The proof utilize the concept of linked and cyclic $G$-sets which are defined in \cite[Def. 3.1, Def 3.3]{Mil}.

\begin{theorem}
\label{thm:diff_fil} Let $G = \ZZ_p^d$ with $d \geq 2$. For each $s \geq 2$, the ideal $\fm^s$ in $\W_{G}(k)$ does not contain any $I_{p^r}$.
\end{theorem}

\begin{proof}
Since $\fm^s \subset \fm^2$ for $s \geq 2$, it suffices to show 
$\fm^2$ contains no $I_{p^r}$ for $r \geq 1$.
For any $n \geq r$, there are linked $G$-sets $T_n$ and $T_n'$ of size $p^n$ by 
\cite[Lem. 3.4]{Mil}.  Then $\omega_{T_n}(1) \in I_{p^r}$ and 
by \cite[Thm. 4.4]{Mil} $\omega_{T_n}(1) \not\in \fm^2$, so $I_{p^r} \not\subset \fm^2$.
\end{proof}

\begin{corollary}\label{cor:TopNotSame}
For $G = \ZZ_p^d$ and $d \geq 2$ the initial vanishing topology and the topology defined by the maximal ideal of $\W_G(k)$ do not agree. 
\end{corollary}

%The following question remains open. 
%
%{\it Question:} Is $\W_{\ZZ_p^d}(k)$ for $d \geq 2$ complete with respect to the topology defined by its maximal ideal?

\section{Metrics for Witt-Burnside rings with $G = \ZZ_p^d$ and $d \geq 2$}\label{sec:norms}

%\begin{theorem}\label{thm:Fill}
%$I_{p^m}(G,k)I_{p^n}(G,k) \subset I_{p^{m+n}}(G,k)$
%\end{theorem}

Theorem~\ref{thm:Fill} suggests that tracking the largest $n$ for which a nonzero Witt vector 
$\mbf{a}$ lies in $I_{p^n}$ should behave like an algebra norm. We exploit this in particular for the case $G = \ZZ_p^2$ which we remark trivially satisfies Definition~\ref{dff:ratio} as it is abelian. Throughout assume that $k$ is a field of characteristic $p$. 

\begin{definition}
Let $G = \ZZ_p^d$. For nonzero $\mbf{a}$ in $\W_{G}(k)$, define its {\it initial vanishing norm} by 
$||\mbf{a}||_{\IV} = 1/p^n$ where 
$\mbf{a} \in I_{p^n}(G,k) - I_{p^{n+1}}(G,k)$.  
That is, $n \geq 0$ is maximal for the property that
$\mbf{a} \in I_{p^n}(G,k)$, or equivalently $\mbf{a}$ has a nonzero 
coordinate at some $G$-set of size $p^n$ and this $n$ is minimal. 
Set $||\mbf{0}||_{\IV} = 0$.
\end{definition}

Note $||\mbf{a}||_{\IV} \leq 1$ for all $\mbf{a}$, with equality exactly on the units in $\W_{\ZZ_p^d}(k)$. 
In the classical case $d=1$ and $k$ perfect, $||\cdot ||_{\IV}$ is the usual $p$-adic absolute value on 
Witt vectors and this identification is critical in Section~\ref{sec:FunRing}. We now restate Theorem~\ref{thm:Fill} in terms of norms.

\begin{theorem}
\label{thm:normmult}
For $\mbf{a}$ and $\mbf{b}$ in $\W_{\ZZ_p^d}(k)$, 
$$ ||\mbf{a}+\mbf{b}||_{\IV} \leq \max(||\mbf{a}||_{\IV},||\mbf{b}||_{\IV}) \hspace{0.3cm} \text{ and
} \hspace{0.3cm}
||\mbf{a}\mbf{b}||_{\IV} \leq ||\mbf{a}||_{\IV}||\mbf{b}||_{\IV}.$$
\end{theorem}

%\begin{proof}
%Both results are obvious if $\mbf{a}$ or $\mbf{b}$ is $\mbf{0}$, so we may 
%assume $\mbf{a}$ and $\mbf{b}$ are not $\mbf{0}$. 
%If $||\mbf{a}||_{\IV} = 1/p^m$ and 
%$||\mbf{b}||_{\IV} = 1/p^n$ where, without loss of generality, $m \leq n$, then 
%$I_{p^m}(\ZZ_p^d,k)$ contains both $\mbf{a}$ and $\mbf{b}$, so it 
%contains their sum. Therefore $||\mbf{a} + \mbf{b}||_{\IV} \leq 1/p^m = 
%\max(1/p^m,1/p^n)$. 
%
%
%The inequality $||\mbf{a}\mbf{b}||_{\IV} \leq ||\mbf{a}||_{\IV}||\mbf{b}||_{\IV}$ 
%is another way of saying Theorem \ref{thm:Fill}. 
%\end{proof}

\begin{remark}
The submultiplicativity of $||\cdot ||_{\IV}$ is the only place where we made serious use of $\ZZ_p^d$ being abelian, since that is how we know it satisfies the ratio property which is needed in Theorem \ref{thm:Fill}.  To the extent we can prove the ratio property for 
nonabelian pro-$p$ groups, we will get a norm on $\W_G(k)$ for more general pro-$p$ groups $G$.
\end{remark}

We now describe a precise `multiplication by $p$' formula in the case that $G$ is a pro-$p$ group. When applied to $G = \ZZ_p^d$ for $d =2$ gives us more information about the norm. We start with a lemma whose proof is routine and left to the reader. 

\begin{lemma}
\label{lem:divis1}
For any profinite group $G$, and $T,U \in \mcal{F}(G)$, $\varphi_T(T)$ divides $\varphi_T(U)\#T / \# U$. 
\end{lemma}

\begin{theorem}
\label{thm:pxx}
Let $G$ be a pro-$p$ group and $R = \ZZ[\underline{X}]$. For $n \geq 0$, let $\mbf{x} \in \W_G(R)$ satisfy $x_T = 0$ when $\# T < p^n$ and $x_T = X_T$ otherwise. 
Set $\mbf{y} = p\mbf{x}$ in $\W_G(R)$. Then 
\begin{displaymath}
y_T =
\begin{cases}
0 & \text{ if $\# T < p^n,$} \\
pX_T & \text{if $\# T = p^n$,} \\
 \sum\limits_{\stackrel{U<T}{\# U = p^n}} \frac{p \varphi_T(U)}{\varphi_T(T)}(1-p^{p-1}) & \text{if $\# T = p^{n+1}$.  }\\
\end{cases}
\end{displaymath}
Reducing modulo $pR$ this gives 
\begin{displaymath}
y_T \equiv
\begin{cases}
0 \bmod pR & \text{ if $\# T \leq p^n,$} \\
\left(\sum\limits_{\stackrel{U<T}{\# U = p^n}} \frac{p\varphi_T(U)}{\varphi_T(T)}X_U\right)^p \bmod pR & \text{if $\# T = p^{n+1}$.  }\\
\end{cases}
\end{displaymath}
\end{theorem}

Note this theorem is not specifying $y_T$ for all $T$, but only for $\# T \leq p^{n+1}$. 

\begin{proof}
For any $T \in \mcal{F}(G)$, 
the equation $\mbf{y} = p\mbf{x}$ becomes, under the $T$th Witt polynomial, 
\begin{equation}\label{xyu}
\sum\limits_{U \leq T} \varphi_T(U) y_U^{\#T/\#U} = p \sum_{U \leq T} \varphi_T(U) x_U^{\#T/\#U}.
\end{equation}
Using induction one can show that $y_T \equiv 0 \bmod pR$ for all $T$ such that $\# T< p^n$ since $x_T = 0$ for all such $T$.
Let $\#T = p^n$. Since $x_U = y_U = 0$ for all $\# U < p^n$, (\ref{xyu}) becomes $\varphi_T(T) y_T = p \varphi_T(T) X_T$, which gives $y_T = p X_T$. 
Let $\# T = p^{n+1}$. In this case (\ref{xyu}) becomes
\begin{equation}
\label{eq:xyu2}
\sum\limits_{\stackrel{U < T}{\#U  = p^n}} \varphi_T(U) y_U^p + \varphi_T(T)y_T = 
p\sum\limits_{\stackrel{U < T}{\#U  = p^n}} \varphi_T(U) x_U^p + \varphi_T(T)x_T.
\end{equation}

We have already shown that $y_U = pX_U$ when $\# U = p^n$ and making this replacement in (\ref{eq:xyu2}) we have

$$
\sum\limits_{\stackrel{U < T}{\#U  = p^n}} p^p\varphi_T(U)X_U^p + \varphi_T(T)y_T = 
\sum\limits_{\stackrel{U < T}{\#U  = p^n}} p\varphi_T(U)X_U^p + p\varphi_T(T)X_T.
$$
By Lemma \ref{lem:divis1}, $\varphi_T(T) |  p\varphi_T(U)$ so $p\varphi_T(U)/\varphi_T(T) \in \ZZ$ and 
\begin{eqnarray*}
y_T & = & \sum_{\stackrel{U < T}{\#U = p^n}} \frac{p\varphi_T(U)}{\varphi_T(T)}(1-p^{p-1})X_U^p + pX_T \\
 &  \equiv & \left( \sum_{\stackrel{U < T}{\#U = p^n}} \frac{p\varphi_T(U)}{\varphi_T(T)}X_U \right)^p \bmod pR.
\end{eqnarray*}
\end{proof}

\begin{remark}
When $G$ is abelian, $\varphi_T(U) = \# U$ and $\varphi_T(T) = \# T$ so 

\noindent $p\varphi_T(U)/\varphi_T(T) = 1$ since $\# T = p^{n+1}$ and $\# U = p^n$. 
\end{remark}

\begin{theorem}\label{thm:corn}
Let $G = \ZZ_p^2$ and $k$ be a field of characteristic $p$. If $\mbf{a} \in I_{p^n}(G,k) - I_{p^{n+1}}(G,k)$ then 
$$p\mbf{a} \in I_{p^{n+1}}(G,k) - I_{p^{n+2}}(G,k).$$ 
\end{theorem}

\begin{proof}
Since $k$ has characteristic $p$, 
in $\W_G(k)$ the 0-coordinate of $p$ is 0, so 
$p \in I_p(G,k)$.  Therefore we have 
$p\mbf{a} \in I_{p}(G,k)I_{p^n}(G,k) \subset I_{p^{n+1}}(G,k)$. 
We will find a nonzero coordinate of $p\mbf{a}$ which is 
indexed by a $G$-set of order $p^{n+1}$.

Since $\mbf{a} \in I_{p^n}(G,k)$, for some $T_0 \in \mcal{F}(G)$ with $\# T_0 = p^n$, $a_{T_0} \neq 0$. Choose such $T_0$ of minimal level. 
For any $T < T_0$, $\#T < p^n$, so $a_{T} = 0$. 
Therefore $\mbf{a}$ is a homomorphic image of $\mbf{x}$ from Theorem \ref{thm:pxx} using functoriality with a suitable 
homomorphism $\ZZ[\underline{X}] \rightarrow k$ and for any $T$ with size $p^{n+1}$, $$(p\mbf{a})_T = \left( \sum_{\stackrel{U<T}{\# U = p^n}} a_U \right)^p.$$  Which $T$ should we look at to find nonzero $a_T$? By \cite[Lem. 5.8]{Mil}, there are covers $T$ of $T_0$ with the same level as $T_0$. Recall, $\Lev(T)$ denotes the level of $T$. For any $T$ covering $T_0$ with $\Lev(T) = \Lev(T_0)$ and any $U < T$ with $\# U = p^n$ and $\Lev(U) < \Lev(T_0)$, $a_U = 0$. So 
$$
(p\mbf{a})_T = \left(\sum_{\stackrel{U < T}{ \stackrel{\#U = p^n}{\Lev(U) = \Lev(T_0)} }} a_U\right)^p.
$$
Since $G = \ZZ_p^2$, \cite[Lem. 5.9]{Mil} says $T_0$ is the only $G$-set below $T$ with the same level and of size $p^n$, hence $(p\mbf{a})_T = a_{T_0}^p \neq 0$. 
\end{proof}

We have a number of useful corollaries. 

\begin{corollary}
When $G = \ZZ_p^2$ and $k$ is a field of characteristic $p$, 
$p$ is not a zero-divisor in $\W_G(k)$. 
\end{corollary}

\begin{proof}
Any nonzero element of $\W_G(k)$ belongs to 
some $I_{p^n}(G,k) - I_{p^{n+1}}(G,k)$, and 
Theorem \ref{thm:corn} shows its product with $p$ is nonzero. 
\end{proof}

\begin{corollary}
\label{cor:pmultsup}
Let $G = \ZZ_p^2$ and $\mbf{x} \neq \mbf{0}$ in $\W_G(k)$. The support of $p\mbf{x}$ contains more than one $G$-set. 
\end{corollary}
\begin{proof}
Set $\mbf{y} = p \mbf{x}$. By \cite[Lem. 5.10]{Mil} there is a $T_0$ such that $T_0$ has minimal size among the nonzero coordinates of $\mbf{x}$. Choose such $T_0$ with minimal level. From the proof of Theorem \ref{thm:corn}, for any $T$ covering $T_0$ with $\Lev(T) = \Lev(T_0)$, $y_T \neq 0$. 
From \cite[Thm. 5.4]{Mil} there is more than one such $T$. 
\end{proof}

Maybe the most important consequence for this article is the following. 

\begin{corollary}
When $G = \ZZ_p^2$, 
the initial vanishing topology on $\W_G(k)$ induces on 
$\ZZ$ its $p$-adic topology.
\end{corollary}
\begin{proof}
By induction with Theorem \ref{thm:corn}, 
starting with $\mbf{a} = (1,0,0,0,\dots)$, 
$p^n \in I_{p^n}(G,k) - I_{p^{n+1}}(G,k)$.  
If $m \in \ZZ$ is not a multiple of $p$ then 
$m \in \W_G(k)^\times$, so 
$p^nm \in I_{p^n}(G,k) - I_{p^{n+1}}(G,k)$.  
Therefore when $a$ and $b$ are in $\ZZ$, 
$a \equiv b \bmod I_{p^n}(G,k)$ 
if and only if $a \equiv b \bmod p^n$. 
\end{proof}

\begin{theorem}
\label{thm:cpt}
Let $G = \ZZ_p^2$. For $c \in \ZZ_p$ and $\mbf{a}$ in $\W_G(k)$, 
$||c \mbf{a}||_{\IV} = |c|_p ||\mbf{a}||_{\IV}$, so $||\cdot ||_{\IV}$ is a $\ZZ_p$-algebra norm 
on $\W_G(k)$.
\end{theorem}

\begin{proof}
This is obvious if $c$ or $\mbf{a}$ vanishes, so we may assume neither does. 
Then we can write $c = p^nu$ for $n \geq 0$ and $u \in \ZZ_p^\times$, so 
it suffices to check the two cases $c = u$ and $c = p$. 

By Theorem \ref{thm:normmult}, for $u \in \ZZ_p^\times$ we have 
$||u\mbf{a}||_{\IV} \leq ||u||_{\IV}||\mbf{a}||_{\IV} \leq ||\mbf{a}||_{\IV}$. Since $u$ is a unit, 
$||\mbf{a}||_{\IV} = ||u^{-1}(u\mbf{a})||_{\IV} \leq ||u^{-1}||_{\IV}||u\mbf{a}||_{\IV} \leq ||u\mbf{a}||_{\IV}$, so $||u\mbf{a}||_{\IV} = ||\mbf{a}||_{\IV}$.

It remains to show $||p\mbf{a}||_{\IV} = (1/p)||\mbf{a}||_{\IV}$, which is 
equivalent to showing that if $\mbf{a} \in I_{p^n}(G,k) - I_{p^{n+1}}(G,k)$ then 
$p\mbf{a} \in I_{p^{n+1}}(G,k) - I_{p^{n+2}}(G,k)$. 
That is Theorem \ref{thm:corn}.
\end{proof}

\begin{remark}
If one shows Theorem \ref{thm:corn} holds for $G = \ZZ_p^d$, $d \geq 3$, Theorem \ref{thm:cpt} will hold as well. 
\end{remark}

Since $p$ is not a zero-divisor in $\W_G(k)$, 
we can formally invert $p$ and consider the ring $\W_G(k)[1/p]$, 
which contains $\ZZ_p[1/p] = \QQ_p$.

\begin{theorem}
Let $G = \ZZ_p^2$. There is a unique extension of $||\cdot||_{\IV}$ from 
$\W_G(k)$ to a $\QQ_p$-algebra norm on $\W_G(k)[1/p]$, described 
as follows. 
For $\mbf{v} \in \W_G(k)[1/p]$, write $\mbf{v} = (1/p^m)\mbf{a}$ where
$m \geq 0$ and $\mbf{a} \in \W_G(k)$.   
Then $||\mbf{v}||_{\IV} = ||\mbf{a}||_{\IV}/|p|_p^m = p^m||\mbf{a}||_{\IV}$. 
\end{theorem}

\begin{proof}
We check the function $||\cdot||_{\IV}$ on $\W_G(k)[1/p]$ is 
well-defined.  If 
$\mbf{v} = (1/p^m)\mbf{a} = (1/p^n)\mbf{b}$ for some 
$m$ and $n$ that are $\geq 0$ and 
$\mbf{a}$ and $\mbf{b}$ in $\W_G(k)$ then 
$p^n\mbf{a} = p^m\mbf{b}$ in $\W_G(k)$. 
Taking the norm of both sides, 
$$
\frac{1}{p^n}||\mbf{a}||_{\IV} = \frac{1}{p^m}||\mbf{b}||_{\IV}, 
$$
so $p^m||\mbf{a}||_{\IV} = p^n||\mbf{b}||_{\IV}$.

By scaling two elements of $\W_G(k)[1/p]$ into $\W_G(k)$ 
by a common power of $p$, doing a computation there, and
then rescaling back, the axioms of a $\QQ_p$-algebra norm 
on $\W_G(k)[1/p]$ are verified for $||\cdot||_{\IV}$

It is simple to see the formula for $||\cdot||_{\IV}$ is the only possible 
$\QQ_p$-algebra norm on $\W_G(k)[1/p]$ extending the original norm on $\W_G(k)$.
\end{proof}

Obviously $\W_G(k)$ is inside the unit ball in $\W_G(k)[1/p]$, but that is 
not the whole unit ball.  For example, 
if $\#T = p^n$ then $||\omega_T(1)||_{\IV} = 1/p^n = ||p^n||_{\IV}$, so 
$(1/p^n)\omega_T(1)$ is in the unit ball but let's see it is not in 
$\W_G(k)$ for $n \geq 1$. If $p^{-n}\omega_T(1) = \mbf{v}$ is in $\W_G(k)$ then $\omega_T(1) = p^n\mbf{v}$, 
however this is not possible as $p^n \mbf{v}$ has a support of more than one $G$-set by 
Corollary \ref{cor:pmultsup} and $\omega_T(1)$  has support $\{ T \}$. 

Although every individual element of $\W_G(k)[1/p]$ can be scaled by a power of $p$ to become an element of $\W_G(k)$, this does not mean we can scale all the terms in a Cauchy sequence in $\W_G(k)[1/p]$ by a single power of $p$ to land in $\W_G(k)$.  

%%There may be a problem with this
\begin{example}
For $n \geq 0$, let $T_{2n} \in \mcal{F}(G)$ have size $p^{2n}$, so 
$||\omega_{T_{2n}}(1)||_{\IV} = 1/p^{2n}$.  Then 
$||\omega_{T_{2n}}(1)/p^n||_{\IV} = 1/p^n \rightarrow 0$, so 
if $\W_G(k)[1/p]$ were complete then 
the series
$$
s = \sum_{n \geq 0} \frac{1}{p^n}\omega_{T_{2n}}(1) 
$$
would converge in $\W_G(k)[1/p]$. Multiplying $s$ by a suitable power of 
$p$ would put it in $\W_G(k)$, say $p^Ns \in \W_G(k)$:
$$
\sum_{n \geq 0} \frac{p^N}{p^n}\omega_{T_{2n}}(1)  \in \W_G(k). 
$$
This is impossible. If $\sum_{n \geq 0} p^{N-n}\omega_{T_{2n}}(1)  \in \W_G(k)$ write $$\sum_{n \geq 0} p^{N-n}\omega_{T_{2n}}(1) = \sum_{n \leq N} p^{N-n}\omega_{T_{2n}}(1) + \sum_{n > N}p^{N-n}\omega_{T_{2n}}(1).$$ Since $\sum_{n \leq N} p^{N-n}\omega_{T_{2n}}(1) \in \W_G(k)$, $\sum_{n > N}p^{N-n}\omega_{T_{2n}}(1)$ would have to be also. However, we would then have a partial sum $p^{-1}\omega_{T_{2(N+1)}}(1) + \ldots +p^{-r}\omega_{T_{2(N+r)}}(1) \in \W_G(k)$ for some $r > 1$ and multiplying by $p^{r-1}$ gives that $p^{-1}\omega_{T_{2(N+r)}}(1)$ is in $\W_G(k)$ which is a contradiction. 
\end{example}

%\subsection{Trees}
%
%
%\begin{definition}\label{def:size}
%Denote by $\tree$ the tree of $\ZZ_p$-lattices in $\QQ_p^2$ up to scaling by $\ZZ_p$\footnote{One needs to say scaling by what.}. 
%\end{definition}
%
%View $\tree$ as a rooted tree constructed as $p+1$ copies of a $p$-ary tree glued together with a  common root. The boundary is $\mathbf{P}^1(\QQ_p)$ \cite{Ser80}. Vertices of the tree are denoted by capital roman letters, usually $T,U,V$. The root node will be denote $\root$. 
%
%\begin{definition} For a vertex $T \in \tree$, the size\footnote{This is how it was defined in the algebra land, it could just as easily been $n$ however, this seems arbitrary and to link up with other things this normalization might be easier. Let's keep this flagged for deletion if we end up not using the size concept.} $\# T$ is $p^n$ where $n$ the number of edges in a path connecting $\root$ and $T$. 
%\end{definition}
%
%Background will describe the ring $\W_{\ZZ_p^2}(k)$ and the ideal $J \subset \W_{\ZZ_p^2}(k)$. The ring structure is not easy to write down as it is defined by recursively defined polynomials. In the quotient $\W_{\ZZ_p^2}(k)/J$ these polynomials are the same as those defining the $p$-typical Witt vectors $W(k)$. These, as a set, are the admissible vectors on $\tree$. 
%
%\begin{definition}\label{def:R} Set $R = \W_{\ZZ_p^2}(k)/J$. Its elements are denoted $\mbf{a} = (a_T)_{T \in \tree}$ where $a_T \in k$. \end{definition}
%
%
%

\begin{remark}
Since the initial vanishing topology on the $p$-typical Witt vectors is exactly its $p$-adic topology, the associated norm $|| \bullet ||_{\IV}$ on the ring $\W_{\ZZ_p}(k)$ is precisely the usual norm, which we denote by $|| \bullet ||_{\W(k)}$, on the Witt vectors. 
\end{remark}

\section{Interpretation as functions rings}\label{sec:FunRing}

One of the most vexing issues in studying Witt-Burnside rings, in particular the $G = \ZZ_p^d$ for $d \geq 2$ cases, is that they have not yet been expressed in terms of any known rings. We summarize what is known both from \cite{Mil} and this article about these cases. The rings $\W_{\ZZ_p^d}(k)$ where $k$ is a field of characteristic $p$ have a unique maximal ideal (\cite[Thm. 2.16]{Mil}) and thus are not noetherian rings (\cite[Thm. 4.5]{Mil}). They are complete both in their natural profinite topology and a metric topology (Theorem \ref{thm:cpt}) which we called the initial vanishing topology.  The initial vanishing topology is equivalent to the profinite topology (Theorem \ref{thm:sametop}). These rings are not domains, but when $d=2$, $\W_{\ZZ_p^2}(k)$ is reduced (\cite[Thm. 5.17]{Mil}). The natural prime ideals of the form $\fp_f := \ker ( f : \ZZ_p^d \onto \ZZ_p)$ have a non-zero intersection; see Subsection~\ref{sec:PI}.  So when $d=2$ there must be more prime ideals. The ring being reduced and non-noetherian suggests
an interpretation as a ring of continuous functions on some topological space, which we give for a suitable quotient of $\W_{\ZZ_p^2}(k)$.

Set $J = \bigcap_{f : \ZZ_p^d \onto \ZZ_p} \fp_f = \{ \mbf{a} \in \W_{\ZZ_p^d}(k) \colon a_T = 0 \mbox{ if } T \mbox{ is cyclic}\}$. This is an ideal by Theorem~\ref{thm:integralwitt}. The quotient $\W_{\ZZ_p^2}(k)/J$ remains local and is also not noetherian
% as the proof of 
\cite[Thm. 4.5]{Mil}. %only required the existence of linked pairs of $\ZZ_p^2$-sets which are prevalent in level $0$. 
It is also reduced % as the proof of
 (\cite[Thm. 5.17]{Mil}). % only required finding a nonzero coordinate in the square of a Witt vector, and this coordinate lies in the same level as any non-zero coordinate of minimal level. 
The initial vanishing topology and its corresponding metric naturally descend and so this quotient becomes a complete metric space in the initial vanishing topology. Note that taking the quotient by $J$ amounts to restricting the Witt vectors in $\W_{\ZZ_p^2}(k)$ exactly to the tree of $\ZZ_p$-lattices in $\QQ_p^2$ up to $\QQ_p^\times$-scaling, i.e,. the Bruhat-Tits tree for $\SL_2(\QQ_p)$ \cite{Ser80} and so we interpret them as $k$-valued functions on the vertices of this tree. As such the ring $\W_{\ZZ_p^2}(k)/J$ is not a Witt-Burnside ring $\W_G$ for any profinite group $G$. 

%to $R$, so $(R, || \bullet ||_{IV})$ is a complete metric space in the IV topology.  

\

{\it For the remainder of the section $\tree$ denotes the tree of $\ZZ_p$-lattices in $\QQ_p^2$ up to $\QQ_p^\times$-scaling with a chosen root $\root \in \tree$ corresponding to the $\ZZ_p^2$-set $\ZZ_p^2/\ZZ_p^2$, $$R := \W_{\ZZ_p^2}(k)/J,$$ and $k$ is a field of characteristic $p$.} 

\subsection{Ultrametric Lipschtiz continuous functions}

To identify $R$ as a function space, we deal with two layers of topological spaces. The first layer consists of the spaces $\tree$ and its boundary and the second layer consists of functions spaces, such as the space of continuous functions over the boundary of $\tree$. We retain the notion of size in $\tree$ as the size of the corresponding $G$-set. The elements of $R$ may be thought of as $k$-valued functions on the vertices of $\tree$ and we call them  {\it admissible vectors}. 

We often consider paths in $\tree$ and in this article we only consider paths which are {\it loop-erased}, i.e., paths that never visit the same vertex twice. This hypothesis is assumed but will not be stated at each instance. 

\begin{definition}The tree $\tree$ has a natural metric on its vertices by setting $d_\tree(T,U) = p^{-n}$ where $n$ is the number of edges in common between the two unique paths from $\root$ to $T$ and $\root$ to $U$. 
\end{definition}

\begin{definition} 
The boundary of $\tree$, denoted $\partial \tree$, is the collection of infinite paths beginning at $\root \in \tree$  with the metric $d_{\partial \tree}$, the extension of $d_{\tree}$ to infinite length paths. This metric space is a realization of $\PP^1(\QQ_p)$ with the usual $p$-adic metric it inherits from $\ZZ_p$ \cite{Ser80}. Note that this metric gives $\PP^1(\QQ_p)$ diameter $1$. We equip the set of paths with an extension of $d_{\tree}$ to paths of infinite length and denote this extension as $d_{\P1}$.
\end{definition}

%We will often consider paths in trees. Since a path is a sequence of vertices in the graph such that there is an edge connected consecutive vertices in the path it is quite possible for an arbitrary path to back track over itself. For example, for $T,U$ vertices in a graph that have an edge connecting them we could have as part of the path sequence $\ldots TUT \ldots$. Of course the single vertex $U$ could be replaced with a longer path of finite length. It will often be important that we restrict to \emph{loop-erased paths} which are paths that never visit the same vertex twice. 

%In Remark \ref{rem:typmet} a metric on $p$-typical Witt vectors was defined. It was also commented on that the underlying graph of for a $p$-typical Witt vector was a single infinite path from a root. So it is possible to view an element of $\partial \tree$ together with its labeling as a Witt vector itself. 

The $p$-typical Witt vectors (i.e., Witt vectors for $G = \ZZ_p$) form a metric space under the usual $p$-adic metric, which has an associated norm that we denote by $|| \bullet ||_{\W(k)}$. As such we may consider continuous functions $\PP^1(\QQ_p) \rightarrow \W(k)$. The Lipschitz condition will also be relevant in later discussions.
%we can easily speak about continuous and Lipschitz continuous functions into $\W(k)$. 
%Considering the $p$-typical Witt vectors $\W(k)$ in its natural metric topology (which is the same as its pro-finite topology\footnote{Not yet defined.} and its initial vanishing topology which is defined by a metric induced from a norm $||\bullet ||_{\W(k)}$). This induces a natural metric on functions from $\PP^1(\QQ_p)$ to $\W(k)$ and as such we can define continuous and Lipschitz continuos functions. 

\begin{definition}
Denote by $\C(\PP^1(\QQ_p),\W(k))$ the continuous functions from $(\PP^1(\QQ_p),d_{\PP^1(\QQ_p)})$  to $(\W(k),\| \bullet\|_{\W(k)})$. Denote by $\Lip_\alpha(\PP^1(\QQ_p),\W(k)) \subset \C(\PP^1(\QQ_p),\W(k))$ the Lipschitz continuous functions with Lipschitz constant less than $\alpha$. That is $f \in \Lip_\alpha(\PP^1(\QQ_p),\W(k))$ provided that for all $x,y \in \PP^1(\QQ_p)$
$$||f(x) - f(y)||_{\W(k)} \leq \alpha \cdot d_{\PP^1(\QQ_p)}(x,y).$$
%Denote by $\C(\PP^1(\QQ_p),\W(k))$ to the continuous functions from $\PP^1(\QQ_p)$ equipped with the metric $d_{\P1}$ to  $\W(k)$ equipped with the metric induced by $|| \bullet ||_{\W(k)}$. Denote by $\Lip_\alpha(\PP^1(\QQ_p),\W(k)) \subset \C(\PP^1(\QQ_p),\W(k))$ the Lipschitz continuous functions with Lipschitz constant $\alpha > 0$. That is $f \in \Lip_\alpha(\PP^1(\QQ_p),\W(k))$ provided $$||f(x) - f(y)||_{\W(k)} \leq \alpha \cdot d_{\PP^1(\QQ_p)}(x,y)$$ for all $x,y \in \PP^1(\QQ_p)$. 
\end{definition}

%The function space $\Lip_\alpha(\PP^1(\QQ_p),\W(k))$ is what we identify with $R$. Since $R$ is a metric space we discuss the corresponding metric on $\Lip_\alpha(\PP^1(\QQ_p),\W(k))$.
%
%\begin{definition}
%Let $X$ be a topological space and $Y$ a metric space. For $f \in \Lip(X,Y)$ set the least Lipschitz constant for $f$ to be  $$|f|_L = \inf_{\alpha>0} \left\{ \alpha : f \in \Lip_{\alpha}(X,Y) \right\}.$$ We give $\Lip(X,Y)$ the metric structure defined by $\|f\|_{\Lip} = \|f\|_{\sup} + |f|_L$ where $\| \bullet \|_{\sup}$ is the supremum norm of $f$. 
%\end{definition}
%
%Naturally the continuous functions $\C(X,Y)$ can be considered with just the supremum norm and alternatively one may endow $\Lip(X,Y)$ with just the $\| \bullet \|_{\sup}$ norm. We will be very specific with regards to which norm is being used in each instance. 
%
%We state a few interesting properties of Lipschitz functions into an ultrametric space. The proofs of these are routine and left to the reader. Example \ref{ex:CtsNotLip} however will show that while the Lipschitz property will be useful in the main theorem the Banach space\footnote{What is meant here?} of Lipschitz functions will not be.
%
The function space we are most interested in is $\Lip_{p^{-1}}(\PP^1(\QQ_p),\W(k))$; however we need to discuss two metrics on it. 
%It is $\Lip_\alpha(\PP^1(\QQ_p),\W(k))$ that is the function space in which we have the most interest. %Since $R$ is a metric space under $||\bullet ||_{\IV}$
%However, we will need to discuss two metrics on $\Lip_\alpha(\PP^1(\QQ_p),\W(k))$.

\begin{definition}
Suppose that $Y$ has a norm. For $f \in \Lip(X,Y)$ let $$|f|_L = \inf_{\alpha>0} \left\{ \alpha : f \in \Lip_{\alpha}(X,Y) \right\}.$$ The space $\Lip(X,Y)$ is metrized by $\|f\|_{\Lip} = \|f\|_{\sup} + |f|_L$, where $\| \cdot \|_{\sup}$ is the supremum norm of $f$ as a function taking values in $Y$ and $|f|_L$ is the least Lipschitz constant for $f$. 

One can alternatively endow $\Lip(X,Y)$ with the norm $\| \cdot \|_{\sup}$ and we will be very specific with regards to which norm is being used in each instance. Naturally $\C(X,Y)$ is endowed with the supremum norm. 
\end{definition}

We note for those who have not considered the collection of Lipschitz functions into bounded ultrametric spaces algebraically that unlike their Archimediean counterparts, ultrametic Lipschitz functions do in fact form a ring.

\begin{lemma}\label{prop:LipProp}
Let $(X,d_X)$ be a metric space and $(Y,d_Y)$ an ultrametric topological ring whose metric arises from a norm. Further assume that $Y$ has diameter $1$. For any $\alpha$, the set $\Lip_\alpha (X,Y)$ forms a topological ring pointwise under addition and multiplication. In particular, multiplication by a constant function provides a $Y$-action on $\Lip_\alpha(X,Y)$ 
\end{lemma}

\begin{proof}
We check that $\Lip_\alpha(X,Y)$ is closed under pointwise sums and products and leave the rest of the details to the interested reader. Since $Y$ is a topological ring the metric is compatible with the ring operations and we can write $d_Y(z,w) = d(z-w,0) := |z-w|_Y$.
Let $f,g \in \Lip_\alpha(X,Y)$ and $x,y \in X$ then
\begin{eqnarray*}
	d_Y\left( (f+g)(x),(f+g)(y)\right) &=& |(f+g)(x) - (f+g)(y)|_Y\\
	&\le & \max \left\{ |f(x) - f(y)|_Y , |g(x) - g(y)|_Y\right\}\\
	&\le & \max \left\{ \alpha d_X(x,y) , \alpha d_X(x,y) \right\} = \alpha d_X(x,y).\\
\end{eqnarray*}

Again let $f,g \in \Lip_\alpha(X,Y)$ and $x,y \in X$ we estimate the norm of the product's oscillation two ways and then recombine. First
\begin{eqnarray*}
	d_Y\left( f(x)g(x), f(y)g(y) \right) &=& | f(x)g(x) - f(y)g(y)|_Y\\
	&=& | (f(x)g(x) - f(y)g(x)) - f(y)g(y) + f(y)g(x) |_Y\\
%	&=& | g(x)(f(x)-f(y)) - f(y)(g(y)-g(x)) |_Y\\
	&\le & \max\{ |g(x)|_Y|f(x)-f(y)|_Y, |f(y)|_Y |g(x) - g(y)|_Y \}.
%	&\le & \max \{ |g(x)|_p \max \{|f(x)|_p,|f(y)|_p\} , |f(y)|_p \max \{|g(x)|_p,|g(y)|_p\}\}\\ 
%	&=& Ld_p(x,y)\max\{|g(x)|_p,|f(y)|_p\}
\end{eqnarray*}
Thus
\begin{eqnarray*}
	d_Y\left( f(x)g(x), f(y)g(y) \right) &=& | f(x)g(x) - f(y)g(y)|_Y\\
	&=& | f(x)g(x) - f(x)g(y) - f(y)g(y) + f(x)g(y) |_Y\\
%	&=& | f(x)(g(x)-g(y)) - g(y)(f(y)-f(x)) |_Y\\
	&\le & \max \{ |f(x)|_Y|g(x)-g(y)|_Y, |g(y)|_Y|f(x)-f(y)|_Y \}.
\end{eqnarray*}

Combining these calculations we get
\begin{equation*}
	d_Y\left( f(x)g(x), f(y)g(y) \right) \le \alpha d_X(x,y)\max_{r=x,y}\{ |f(r)|_Y,|g(r)|_Y \} \le \alpha d_X(x,y)
\end{equation*}
Since $f$ and $g$ are $Y$ valued $\max_{r=x,y}\{ |f(r)|_Y,|g(r)|_Y \} \le 1$ so we have that $fg$ is also in  $\Lip_\alpha(X,Y)$.
\end{proof}

We state the completeness of Lipschitz functions into an ultrametric space. Example \ref{ex:CtsNotLip} however will show that while the Lipschitz property will be useful in the main theorem the natural Banach space structure of Lipschitz functions will not be.

\begin{theorem}\label{prop:Lipcomplete}
If $X$ is a metric space and $Y$ is a complete metric space, then $\Lip(X,Y)$ is complete with respect to the $\| \cdot \|_{\Lip}$ norm.
\end{theorem}

\begin{proof} 
The Lipschitz norm dominates the supremum norm by construction, i.e. $\| \cdot \|_{\Lip} \ge \| \cdot \|_{\sup}$. %So a $\| \cdot \|_{\Lip}$-Cauchy sequence is also a Cauchy sequence in the supremum norm, so has a limit. 
Let $f_n$ be a $\| \cdot \|_{\Lip}$-Cauchy sequence. Then because $\C(X,Y)$ is complete under the supremum norm there is a pointwise limit of the $f_n \in \C(X,Y)$, call it $f$. An invocation of the triangle inequality shows that $|f|_L \le \liminf_{n \rightarrow \infty} |f_n|_L$. 

Suppose that $\liminf_{n \rightarrow \infty} |f_n|_L = \infty$. There exists $M$ such that if $n,m \ge M$ then $\|f_n-f_m\|_{\Lip} < \epsilon$ because $f_n$ is $\| \cdot \|_{\Lip}$-Cauchy. Since $\liminf_{n \rightarrow \infty} |f_n|_L = \infty$ there exists $M'$ such that $|f_n|_L > \epsilon$ and then $|f_n - f_m|_L > \epsilon$ for $n \ge \max(M,M')$ and $m \ge M$ fixed. Thus $\|f_n - f_m\|_{\Lip} \ge \epsilon$ for large enough $n$ for each $m$ which contradicts $f_n$ being a $\| \cdot \|_{\Lip}$-Cauchy sequence. So $\liminf_{n \rightarrow \infty} |f_n|_L < \infty$. Hence $\Lip(X,Y)$ is complete under the $\| \cdot \|_{\Lip}$-norm. %Since $f$ both exists as a continuous function and has finite Lipschitz constant.
\end{proof}

\begin{corollary}
For any $\alpha > 0$, $(\Lip_{\alpha}(\PP^1(\QQ_p),\W(k)),\| \bullet \|_{\Lip})$ is a complete $\W(k)$-module.
\end{corollary}

\begin{proof}
It is naturally a module via Lemma~\ref{prop:LipProp}. In the proof of Theorem~\ref{prop:Lipcomplete} the bound on the Lipschitz constant of the limit was the $\liminf$ of the Lipschitz constants of the approximating sequence. Thus if those constants are bounded by $\alpha$ then so is the Lipschitz constant of the limit function.
\end{proof}

The Lipschitz property will be useful in identifying the image of $R$ under the mapping that is defined next. However due to Example~\ref{ex:CtsNotLip} when the supremum norm is used on Lipschitz functions it yields a genuinely different metric space which is more correctly thought of as a subspace of continuous functions. 
%The implication of this is that while the Lipschitz condition is an useful way to identify the images of $R$ under $\Phi$ as a subspace of the continuous functions it provides no additional insight into the metric structure of $R$ in the initial vanishing topology.\

The ring structure of $R$ really is very much like a countable number of copies of the $p$-typical Witt vectors glued together long coordinates identified by the paths in $\tree$. To make this more precise we have the following lemma. This extension is well-defined since any two paths which are distance zero from each other must share all vertices, that is they must be the same path.

\begin{lemma}\label{lem:WittPolyAgree}
Let $T \in \tree$ with size $\# T = p^n$ and let $T_0, T_1, \ldots, T_n = T$ be a path from $\root$ to $T$ in $\tree$. The sum and product polynomials $S_T$ and $M_T$ defining the ring structure on $R$ agree with the sum and product polynomials of the classic $p$-typical Witt vectors up to a relabeling of variables.%$n$-th Witt sum 
\end{lemma}

\begin{proof}
Since $T$ is a cyclic $\ZZ_p^2$-set of size $p^n$, it is isomorphic to $\ZZ/p^n \ZZ$ and the $\ZZ_p^2$-sets $U \leq T$ are isomorphic to  $\ZZ/p^r \ZZ$ with $0 \leq r \leq n$. Identify $T_i = \ZZ/p^i \ZZ$ for $i = 0, \ldots, n$.  The Witt polynomial $W_T$ depends only on the variables $X_{T_i}$ for $i = 0, \ldots, n$ and by definition is $W_T(X_{T_0}, \ldots, X_{T_{n-1}}, X_T) = X_{T_0}^{p^n} + p X_{T_1}^{p^{n-1}} + \cdots + p^n X_T$ which is up to a change of variables the classic $p$-typical $n$-th Witt polynomial. Therefore $S_T$ and $M_T$ are also the classic $p$-typical Witt sum and product polynomials up to the same relabeling of variables. 
%\footnote{This proof seems to be nothing more than citing definitions. Is there something more going on here? --BAS} \footnote{No there isn't but this is a key enough observation that listed it as lemma. Maybe it should be just a remark? -- LEM}
\end{proof}

\begin{definition} 
Define $\Phi \colon R \to \Lip_{p^{-1}}(\PP^1(\QQ_p),\W(k))$ by setting $\Phi(\mbf{a}) := f_\mbf{a}$ where the function $$f_{\mbf{a}} \colon \PP^1(\QQ_p) \to \W(k),$$ has value $f_{\mbf{a}}(x)$ defined by taking the unique rooted path $\root, T_1, T_2, \ldots$ in $\tree$ and setting $$f_{\mbf{a}}(x) = (a_{T_0}, a_{T_1}, a_{T_2}, \ldots) \in \W(k).$$ 
\end{definition}

\begin{theorem}\label{thm:isomorphism} 
The ring $R$  is isomorphic to $\Lip_{p^{-1}}(\PP^1(\QQ_p),\W(k))$.
\end{theorem}

\begin{proof}
We start by an association between $R$ and $\W(k)$-valued functions on $\PP^1(\QQ_p)$. In particular, for $\mbf{a} \in R$ set $\Phi(\mbf{a}) := f_\mbf{a}$ where the function $f_{\mbf{a}} \colon \PP^1(\QQ_p) \to \W(k),$ has value $f_{\mbf{a}}(x)$ defined by taking the unique rooted path $\root, T_1, T_2, \ldots$ in $\tree$ and setting $f_{\mbf{a}}(x) = (a_{T_0}, a_{T_1}, a_{T_2}, \ldots) \in \W(k)$. 

By Lemma \ref{prop:LipProp} $\Lip_{p^{-1}}(\PP^1(\QQ_p),\W(k))$ is a ring.  When $x$ and $y$ are in $\PP^1(\QQ_P)$ and $d_{\PP^1(\QQ_p)}(x,y) = p^{-n}$, the paths realizing $x$ and $y$ share the first $n$ edges. This means the first $n+1$ components of $f_{\mbf{a}}(x)$ and $f_{\mbf{a}}(y)$ agree, hence 
	$$||f_{\mbf{a}}(x) - f_{\mbf{a}}(y)||_{\W(k)} \leq p^{-(n+1)} = {1 \over p}d_{\PP^1(\QQ_p)}(x,y),$$
 i.e., the function $f_{\mbf{a}} \in  \Lip_{p^{-1}}(\PP^1(\QQ_p),\W(k))$ for $\mbf{a} \in R$. It follows from Lemma~\ref{lem:WittPolyAgree} that $\Phi$ is a homomorphism. 

When $\mbf{a} \neq 0$ then $a_T \neq 0$ for some $T \in \tree$. Therefore, if $x \in \PP^1(\QQ_p)$ is any path in $\tree$ containing $T$ then $\Phi(\mbf{a})(x)$ will not be zero, hence $\Phi$ is injective.  

Given $f \in \Lip_{p^{-1}}(\PP^1(\QQ_p),\W(k))$ and $x$ and $y$ in $\PP^1(\QQ_p)$, $||f(x) - f(y)||_{\W(k)} \leq p^{-1}$ as the diameter of $\PP^1(\QQ_p)$ is $1$. Therefore, we can construct $\mbf{a}_f \in R$ by setting $(\mbf{a}_f)_T = f(x)_{i}$ where $T_0, T_1, \ldots, T_i = T$ is the unique rooted path in $\tree$ from $T_0$ to $T$ and $x$ is any point in $\PP^1(\QQ_p)$ extending this path. This is well defined precisely because $f \in \Lip_{p^{-1}}(\PP^1(\QQ_p),W(k))$. In particular, if $y \in \PP^1(\QQ_p)$ is any other path containing $T_0, T_1, \ldots, T_i = T$ then $f(y)_i = f(x)_i$ as $||f(x) - f(y)||_{\W(k)} \leq {1 \over p}d_{\PP^1(\QQ_p)}(x,y) \leq {1 \over p} p^{-i}$ and so $f(y)_i = f(x)_i$. By construction, $\mbf{a}_f$ is a preimage under $\Phi$ for $f$. 
\end{proof}

%Considering $\Lip_{p^{-1}}(\PP^1(\QQ_p),\W(k))$ as a subring of $\C(\PP^1(\QQ_p), \W(k))$ it inherits a metric defined by the sup-norm $|| \bullet ||_{\sup}$, where $|| f ||_{\sup} = \sup_{x \in  \PP^1(\QQ_p)} \{ || f(x)||_{\W(k)} \}$.  

\begin{theorem}\label{thm:Isometry}
The ring isomorphism $\Phi \colon (R,\|\bullet\|_{\IV}) \rightarrow (\Lip_{p^{-1}}(\PP^1(\QQ_p),\W(k)),\| \bullet \|_{\sup})$ from Theorem~\ref{thm:isomorphism} is an isometry.
%The map $\Phi \colon R \to \Lip_{p^{-1}}(\PP^1(\QQ_p),\W(k))$ is an isometry where $R$ is given its initial vanishing metric and $\Lip_{p^{-1}}(\PP^1(\QQ_p),\W(k))$ is given the metric defined by the sup-norm. 
\end{theorem}

\begin{proof}
Given $\mbf{a}$ and $\mbf{b}$ in $R$, it suffices to show $|| \mbf{a} - \mbf{b} ||_{\IV} = || \Phi(\mbf{a}) - \Phi(\mbf{b}) ||_{\sup}$. Suppose $|| \mbf{a} - \mbf{b} ||_{\IV} = p^{-n}$. Writing $\mbf{a} = (a_T)_{T \in \tree}$ and $\mbf{b} = (b_T)_{T \in \tree}$ we have $a_T = b_T$ for any $T \in \tree$ with $\# T = p^i$ and $i \leq n$. For any $x \in \PP^1(\QQ_p)$, realizing $x$ as a path $T_0,T_1,T_2,\ldots$ then $\Phi(\mbf{a})(x)_{T_i} = \Phi(\mbf{b})(x)_{T_i}$ for $i = 0, \ldots, n$.  Therefore we have $|| \Phi(\mbf{a})(x) - \Phi(\mbf{b})(x)||_{\W(k)} \leq p^{-n}$ and so $|| \Phi(\mbf{a}) - \Phi(\mbf{b}) ||_{\sup} \leq || \mbf{a} - \mbf{b} ||_{\IV}$. 

That $|| \mbf{a} - \mbf{b} ||_{\IV} = p^{-n}$, means there is $T \in \tree$ with $\# T = p^{n+1}$ and $a_T \neq b_T$. Pick any $x \in \PP^1(\QQ_p)$, realized as a path $T_0, T_1, T_2, \ldots$ where $T_{n+1} = T$. Note that $\Phi(\mbf{a})(x)_{T_{n+1}} \neq \Phi(\mbf{b})(x)_{T_{n+1}}$ and so $|| \Phi(\mbf{a})(x) - \Phi(\mbf{b})(x)||_{\W(k)} = p^{-n}$. Therefore, we have also that $|| \Phi(\mbf{a}) - \Phi(\mbf{b}) ||_{\sup} = || \mbf{a} - \mbf{b} ||_{\IV},$ as desired. 
\end{proof}

\begin{corollary}\label{cor:maintheorem} 
As metric-topological rings, $(R,\|\bullet \|_{\IV})$ and $(\Lip_{1 \over p}(\PP^1(\QQ_p), \W(k)),\|\bullet \|_{\infty})$ are isomorphic and isometric. 
\end{corollary}

%\begin{proof}
%This follows directly from Theorem~\ref{thm:isomorphism} and Theorem~\ref{thm:Isometry}. 
%\end{proof}

\begin{example}\label{ex:CtsNotLip}
Let $T_0,T_1,T_2,\ldots$ be a rooted, infinite path in $\tree$. Let $a_{T_i}$ be the admissible vector which zero everywhere except for $T_i$ where it is $1$. Consider $f_i = f_{a_{T_i}}$. For all $i \ge 1$ $|f_i|_L = p^{-1}$ and $\|f_i\|_{\infty} = p^{-i}$. In the supremum norm this sequence of functions on $\B{P}^{1}(\B{Q}_p)$ converges to the zero function. However $\|f_i\|_{\Lip} \rightarrow p^{-1}$, so this sequence does not converge as a sequence of Lipschitz functions. 

Let $\tree_i = \{ T \in \tree :\ \#T = p^{i}\}$. Let $a_i$ be $1$ on $\tree_i$ and zero elsewhere and $f_i = f_{a_i}$. Then $|f_i|_L = 0$ and $\| f_i \|_{\infty} = p^{-i}$. This sequence of functions converges in both the Lipschitz and uniform senses to the zero function. 

In fact these examples show that $$\| \cdot \|_{\infty} \le \| \cdot \|_{\Lip} \le \| \cdot \|_{\infty}+ \frac{1}{p}.$$ 
\end{example}

\begin{remark}\label{rem:typmet} Consider the arguments in this section for the $p$-typical Witt vectors in place of $R$. In this case the underlying graph is not a $p+1$-regular tree but a rooted chain whose boundary is a single point $x_0$ and the `functions' one obtains by proving Theorem~\ref{thm:Isometry} are $\Lip_{p^{-1}}(\{x_0\}, \W(k)) \cong \W(k)$. 
\end{remark}

\section{Applications}

\subsection{Krull Dimension of $\W_{\ZZ_p^d}(k)$ for $d \geq 2$} There are non-noetherian rings of finite dimension. For example, every boolean ring has Krull dimension zero, however the boolean ring $\C(X, \ZZ/2 \ZZ)$ can fail to be  noetherian.  In fact, for $X$ a compact, totally disconnected, and Hausdorff space, $\C(X, \ZZ/2\ZZ)$ is only noetherian when $X$ is finite. Likewise, there are well-known examples from Nagata of infinite dimensional noetherian rings \cite{Nag62}. Our first application shows that $\W_{\ZZ_p^d}(k)$ for $d \geq 2$ has infinite Krull dimension.
%\footnote{You start with a rather strong claim here about what is expected and then show that there are straightforward known counter examples. I am not sure this is the best that this paragraph can be. -- BAS} 
%\footnote{ How is this? -- LEM}

To describe the dimension of $\W_{\ZZ_p^d}(k)$  we must have a good understanding of its prime ideals. From the discussion in Section~\ref{sec:PI}, we already know an incomplete but natural collection of prime ideals in $\W_{\ZZ_p^2}(k)$. 
%but that they are not all of them. The reason we know there are more does not apply to the quotient $R$, as this natural family of ideals now has zero intersection.
%\footnote{This follows directly from how $R$ is defined, in particular $R$ is the quotient of $\W_{\ZZ_p^2}(k)$ by the desired intersection. -- LEM} 
%Are there more prime ideals in $R$? We exploit ultrafilters to prove that $\W_{\ZZ_p^d}(k)$ is indeed infinite dimensional when $d \geq 2$. 
We turn to an ultrafilter construction to find more prime ideals. We review what is necessary here, please see \cite[Ch. 1 and 2]{Sch10} for more details. 

\begin{definition} An ultrafilter $\cU$ on a  set $S$ is a collection of subsets of $S$ which satisfies 
\begin{enumerate}
\item $\emptyset \not\in \cU$, 
\item $\cU$ is closed under intersections and taking supersets, 
\item for any subset $R \subset S$ either $R \in \cU$ or $S \setminus R \in \cU$. 
\end{enumerate}
\end{definition}

The only ultrafilters of a given set $S$ that can be easily written down are the principal ultrafilters which contain a singleton subset of $S$. We are interested in non-principal ultrafilters of $S$, i.e., those that do not contain any finite subset of $S$, which can be shown to exist by Zorn's lemma.  
%\footnote{You had CITE here, this is standard material, but if you want I can find a reference. Other things I've written for non-standard analysis have not given a reference, but just l me know if you really want me to find one before a potential review asks us to.  -- LEM}

The proof presented draws inspiration from a similar proof of the infinite dimensionality of the ring of entire functions on $\CC$, which is described in \cite[App. B]{Osb00}. In fact, Lemma~\ref{lem:ultraprime} is a restatement of \cite[Prop. B.13]{Osb00}, we include his proof here for convenience of the reader. The analogous result to \cite[Cor. B.14]{Osb00} required some modification to become Lemma \ref{lem:funzero}.

It is more convenient to work with the $p$-adic valuation rather than norms. In particular, for a $p$-typical Witt vector $\mbf{a} \in \W(k)$ denote by $\ord \mbf{a} = - \log_p ||\mbf{a}||_{\W(k)}$ and note for any $\mbf{a}, \mbf{b} \in \W(k)$ one has $\ord ( \mbf{a} + \mbf{b}) \geq \max\{ \ord(\mbf{a}), \ord(\mbf{b}) \}$ and $\ord ( \mbf{a} \mbf{b} ) = \ord ( \mbf{a} ) + \ord ( \mbf{b} )$. This may seem odd as the valuation and its associated norm are essentially the same data, however we prefer the valuation as it is not multiplicative which comes up in the proof of Lemma~\ref{lem:ultraprime}. First we show how the primes are constructed. 

\begin{lemma}\label{lem:ultraprime}
Let $\cU$ be a non-principal ultrafilter on $\PP^1(\QQ_p)$ such that there is a countable discrete set $D \in \cU$. Let $\sigma \colon D \to \W(k)$ be any set function. Set $\fp_\sigma \subset \Lip_{p^{-1}}(\PP^1(\QQ_p), \W(k))$ such that there is a set $A \in \cU$ with $A \subset D$ and a positive real number $c \in \RR_{\geq 0}$ such that $$\ord f(z) \geq c \cdot  \ord \sigma(z) \text{ for all } z \in A.$$ The set $\fp_\sigma$ is a prime ideal. 
\end{lemma}
\begin{proof}
Suppose $f,g \in \fp_\sigma$ and let $A \in \cU$ be such that $A \subset D$ and let $c > 0$ be such that $\ord f(z) \geq c \ord \sigma(z)$ for all $z \in A$. Also set $B \subset D$ with $B \in \cU$ and $d > 0$ such that $\ord g(z) \geq d \ord \sigma(z)$ for all $z \in B$. Note that $A \cap B \in \cU$ and $A \cap B \subset D$. Since $\ord(f + g) \geq \max\{\ord(f), \ord(g)\}$ we have that $\ord(f(z)+g(z)) \geq \max\{c,d\} \ord \sigma(z)$ for all $z \in A \cap B$. 

Now suppose that $h \in \Lip_{p^{-1}}(\PP^1(\QQ_p), \W(k))$ and $f \in \fp_\sigma$. Let $A \in \cU$ be such that $A \subset D$ and let $c > 0$ be such that $\ord f(z) \geq c \ord \sigma(z)$ for all $z \in A$. Since at any point,  $\ord(hf) = \ord(h) + \ord(f) \geq \ord(f)$ we have that $\ord(h(z)f(z)) \geq c \ord(\sigma(z))$ for all $z \in A$. So $\fp_\sigma$ is an ideal. 

Finally, suppose $fg \in \fp_\sigma$ and $A \in \cU$ be such that $A \subset D$ and let $c > 0$ be such that $\ord f(z) g(z) \leq c \ord \sigma(z)$ for all $z \in A$. Consider the set $B = \{ z \in A \colon \ord f(z) \geq \ord g(z) \}$. When $B \in \cU$, we have for any $z \in B$,  $\ord(f(z)g(z)) = \ord(f(z)) + \ord(g(z)) \geq 2\ord(g(z))$. Thus, for any $z \in A \cap B$ we have $2 \ord(g(z)) \geq c \ord( \sigma(z))$ and so $g \in \fp_\sigma$. A similar argument shows that $f \in \fp_{\sigma}$ when $\PP^1(\QQ_p) \setminus B \in \cU$. Thus $\fp_\sigma$ is a prime ideal. 
\end{proof}

The rest of the argument rests on specifying a countable family of functions $\sigma_i$ so that the associated primes form an infinite chain. To show the inclusions involved are strict, we need to guarantee the existence of particular Lipschitz functions. 

\begin{lemma}\label{lem:funzero}

%For each $\ell > 0$, there is a countable discrete set  $D = \{d_1, d_2, \ldots \} \subset \PP^1(\QQ_p)$ such that for any $i \in \{1,2,\ldots, \ell\}$, there is a function $f \in \Lip_{p^{-1}}(\PP^1(\QQ_p), \W(k))$ such that $f(d_n) = \omega_{p^{n+i}}(1)$.
There is a countable discrete set $E = \{e_1, e_2, \ldots \} \subset \PP^1(\QQ_p)$ such that for any $i \in \NN$, there is a function $f_i \in \Lip_{p^{-1}}(\PP^1(\QQ_p), \W(k))$ with the property that $f_i(e_n) = \omega_{p^{n+i}}(1)$. 

\end{lemma}
\begin{proof} The following construction is illustrated in Figure~\ref{fig:pathsfun}. 
Let $e_1 \in \PP^1(\QQ_p)$. Successively choose $e_n,\ n \ge 2$ in the following manner: let $e_2 \in \PP^1(\QQ_p)$ be such that $d_{\PP^1(\QQ_p)}(e_1,e_2) = p^{-1}$. As paths in $\tree$, $e_1$ and $e_2$ will share exactly one common edge. Choose $e_3 \in \PP^1(\QQ_p)$ so that $d_{\PP^1(\QQ_p)}(e_2,e_3) = p^{-2}$, that is the paths $e_2$ and $e_3$ share two edges while $e_1$ and $e_3$ share a single edge. In general choose $e_n \in \PP^1(\QQ_p)$ such that $d_{\PP^1(\QQ_p)}(e_{n-1},e_n) = p^{-n+1}$ which will force $e_{n-1}$ and $e_n$ to share $n-1$ edges as paths in $\tree$. Denote $E = \{e_n\}_{n \ge 1}$, which is a countable discrete set in $\PP^1(\QQ_p)$ with $||e_n - e_m||_{\PP^1(\QQ_p)} = p^{-\min(n,m)}$ for all $e_n, e_m \in E$.

%Let $m$ and $n$ be natural numbers. Since  $$|| \omega_{p^{n+i}} - \omega_{p^{m+i}}||_{\W(k)} = p^{-\min(n+i,m+i)}$$ it is necessary to pick $D = \{ d_1, d_2, \ldots\} $ so that for each $n$ and $m$ we have $$p^{-\min(n+i,m+i)} \leq {1 \over p}||d_n - d_m||_{\PP^1(\QQ_p)}.$$ Suppose $||d_n - d_m||_{\PP^1(\QQ_p)} = p^{-k}$. We are asking that $k + 1 \leq \min(n+i, m+i) = \min(n,m) + 1$ and this is guaranteed by the construction of $D = \{d_1,d_2,\ldots \}$. 

We now define the function $f_i$. This is easier done using Theorem~\ref{thm:isomorphism} as we can now think of Witt vectors in $R$. To each $e_n$ associate the corresponding path $\root, T_{1,n}, T_{2,n}, \ldots$ in $\tree$ and the distance bound on $E$ means exactly that $T_{ m, m } = T_{m, n}$ for all $m \leq n$ and fixed $m \in \NN$ and these are the only equalities among the paths. Let $\mbf{f}_i = (f_{i,T})_{T \in \tree} \in R$ be defined by setting $f_{i,T} = 1$ when $T = T_{n+i,n}$ and $0$ otherwise. Because $T_{n+i,n}$ only appears in the path $e_n$ and in no other paths then we can conclude that for any infinite path from $\root$ there is exactly one non-zero label on the path.

See Figure~\ref{fig:pathsfun} where the locations of the `$1$' indicates the only nonzero values for $f_i$ in the case when $i = 1$. When $i > 1$ these nonzero values are pushed farther along the respective paths. Since this is in $R$, its image under the map defined in Theorem~\ref{thm:isomorphism} is a function in $\Lip_{p^{-1}}(\PP^1(\QQ_p), \W(k))$ and by construction $f_i(e_n) = \omega_{p^{n+i}}(1)$.

\begin{figure}[t]
\begin{tikzpicture}[inner sep=0pt, scale=0.4pt]
\tikzstyle{dot}=[fill=black,circle,minimum size=4pt]
\tikzstyle{rdot}=[fill=red,circle,minimum size=4pt]
\tikzstyle{bdot}=[fill=blue,circle,minimum size=4pt]
\tikzstyle{gdot}=[fill=darkgreen,circle,minimum size=4pt]
\tikzstyle{pdot}=[fill=magenta,circle,minimum size=4pt]
%[dot/.style={fill=black,circle,minimum size=3pt}],
%rdot/.style={fill=red,circle,minimum size=3pt}]]

%\tikzstyle{dot} = [fill=black,circle,minimum size=3pt]
%rdot/.style={fill=red,circle,minimum size=3pt}]
%bdot/.style={fill=blue,circle,minimum size=3pt}]
%gdot/.style={fill=darkgreen,circle,minimum size=3pt}]

%\node[fill=magenta,circle,minimum size=3pt] (v0) at (0,0) {};

\node[gdot] (v0) at (0,0) {};

\node[gdot] (v1) at (2,3.5) {};
\node[dot] (v2) at (2,0) {};
\node[dot] (v3) at (2,-3.5) {};

\node[gdot] (v4) at (4,4) {};
\node[bdot] (v5) at (4,3) {};

\node at (4,5) {$1$};

\node[dot] (v6) at (4,0.5) {};
\node[dot] (v7) at (4,-0.5) {};

\node[dot] (v8) at (4,-3) {};
\node[dot]  (v9) at (4,-4) {};

\node[gdot] (v11) at (8,5.25) {};
\node[dot] (v10) at (8,4.25) {};

\node[bdot] (v13) at (8,3.25) {};
\node[pdot] (v12) at (8,2.25) {};

\node at (8.5,3.25) {$1$};

\node[dot] (v14) at (8,0.25) {};
\node[dot] (v15) at (8,1.25) {};

\node[dot] (v16) at (8,-0.25) {};
\node[dot] (v17) at (8,-1.25) {};

\node[dot] (v18) at (8,-2.25) {};
\node[dot] (v19) at (8,-3.25) {};

\node[dot] (v20) at (8,-4.25) {};
\node[dot] (v21) at (8,-5.25) {};

\draw[color=darkgreen] (v0) -- (v1);
\draw (v0) -- (v2);
\draw(v0) -- (v3);

\draw [color=darkgreen](v1) -- (v4);
\draw [color=blue](v1) -- (v5);

\draw (v2) -- (v6);
\draw (v2) -- (v7);

\draw (v3) -- (v8);
\draw (v3) -- (v9);

\draw (v4) -- (v10);
\draw [color=darkgreen](v4) -- (v11);

\draw [color=magenta](v5) -- (v12);
\draw [color=blue](v5) -- (v13);

\draw (v6) -- (v14);
\draw (v6) -- (v15);

\draw (v7) -- (v16);
\draw (v7) -- (v17);

\draw (v8) -- (v18);
\draw(v8) -- (v19);

\draw (v9) -- (v20);
\draw (v9) -- (v21);

\fill[black] (9.5,0) circle (0.05cm);
\fill[black] (9.7,0) circle (0.05cm);
\fill[black] (9.9,0) circle (0.05cm);

\draw (11,6) -- (11,-6);

\draw[color=darkgreen,thick] (10.6,5.25) -- (11.4,5.25);

\node at (12,5.25) {$\color{darkgreen} e_1$}; 

\draw[color=magenta,thick] (10.6,2.25) -- (11.4,2.25);

\node at (12,2.25) {$\color{magenta} e_3$};

\draw[color=blue,thick] (10.6,3.25) -- (11.4,3.25);

\node at (12,3.25) {$\color{blue} e_2$}; 

\node at (5,-7.5) {$\tree$};

\node at (12,-7.5) {$\PP^1(\QQ_p)$};

%\path (0:0cm) node (v0) {$v_0$};
%
%\path (30:0.25cm) node (v1) {};
%\path (150:0.25cm) node (v2) {};
%\path (270:0.25cm) node (v3) {};
%
%
%\path (10:0.5cm) node (v4) {};
%\path (50:0.5cm) node (v5) {};
%\path (130:0.5cm) node (v6) {};
%\path (170:0.5cm) node (v7) {};
%\path (250:0.5cm) node (v8) {};
%\path (290:0.5cm) node (v9) {};
%
%\path (0:0.75cm) node (v10) {};
%\path (20:0.75cm) node (v11) {};
%
%
%\path (40:0.75cm) node (v12) {};
%\path (60:0.75cm) node (v13) {};
%
%\path (120:0.75cm) node (v14) {};
%\path (140:0.75cm) node (v15) {};
%
%\path (160:0.75cm) node (v16) {};
%\path (180:0.75cm) node (v17) {};
%
%\path (240:0.75cm) node (v18) {};
%\path (260:0.75cm) node (v19) {};
%
%\path (280:0.75cm) node (v20) {};
%\path (300:0.75cm) node (v21) {};

%\draw (v0) -- (v1)
%(v0) -- (v2)
%(v0) -- (v3)
%(v1) -- (v4)
%(v1) -- (v5)
%(v2) -- (v6)
%(v2) -- (v7)
%(v3) -- (v8) 
%(v3) -- (v9)
%(v4) -- (v10)
%(v4) -- (v11)
%(v5) -- (v12)
%(v5) -- (v13)
%(v6) -- (v14) 
%(v6) -- (v15)
%(v7) -- (v16)
%(v7) -- (v17)
%(v8) -- (v18)
%(v8) -- (v19)
%(v9) -- (v20)
%(v9) -- (v21);

%\node (LT) at (-7,1) {Log Terminal};
%\node (LC) at (-7,-1){Log Canonical};
%\node (R) at (-3,1){Rational};
%\node(DB) at (-3,-1){Du Bois};
% 
%
%\draw[->] (LT) edge[double] (LC);
%\draw[->] (LT) edge[double] (R);
%\draw[->] (R) edge[double] (DB);
%\draw[->] (LC) edge[double] (DB);

\end{tikzpicture}
\caption{The set $E$ and the function $f_i$ with $i = 1$.}
\label{fig:pathsfun}
\end{figure}
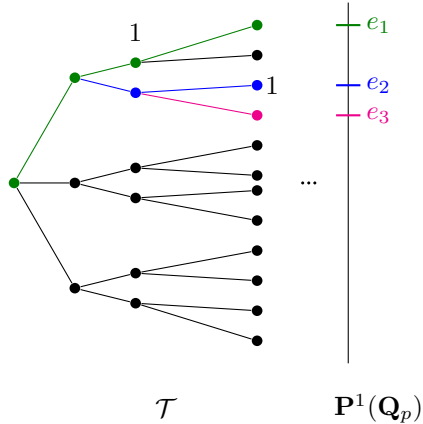

\begin{remark}\label{rmk:conversionlineartoexponential}
The set $E$ depends only on the metric of $\PP^1(\QQ_p)$ and is independent of $i$. The application of Lemma~\ref{lem:funzero} will occur in Theorem~\ref{thm:diminfinite} justifying the existence for any $i \geq 2$, there is a function $f_i$ such that $f_i(e_n) = \omega_{p^{n^i}}(1)$, because $n^i$ may be written as $n + z$ for some positive integer $z$. 
\end{remark}

%Suppose $||f(x) - f(d_n)||_{\W(k)} \leq {1 \over p} || x - d_n||_{\PP^1(\QQ_p)}$ for all $x \in \PP^1(\QQ_p)$ and $d_n \in D$, and for any $x$ and $y$ in $\PP^1(\QQ_p)$ if $||x - y || \leq || x - d_n||$ then $|f(x) - f(y)|| \leq ||f(x) - f(d_n)||$. 
%
%Suppose  $x$ and $y$ are any two points in $\PP^1(\QQ_p)$ then $$|| f(x) - f(y) || = || f(x) - f(d_n) + f(d_n) - f(y) || \leq \max\{ ||f(x) - f(d_n)||, || f(d_n) - f(y) || \}.$$ 
%
%Choose $d_n \in D$ so that $||f(x) - f(y)|| \leq || f(x) - f(d_n)||$ and $|| x - y || \leq ||x - d_n||$. Suppose $|| f(x) - f(y) || \leq ||f(x) - f(d_n)|| \leq {1 \over p} || x - d_n||$ and $||x - y || \leq \max\{ ||x - d_n|| , ||d_n - y|| \}$. 
%

%First set 
%\begin{displaymath}
%g_z^k(x) = 
%\begin{cases}
%0 & \text{ if $\# T < p^n,$} \\
%pX_T & \text{if $\# T = p^n$,} \\
% \sum\limits_{\stackrel{U<T}{\# U = p^n}} \frac{p \varphi_T(U)}{\varphi_T(T)}(1-p^{p-1}) & \text{if $\# T = p^{n+1}$.  }\\
%\end{cases}
%\end{displaymath}

%
%Thus, $f$ is in $\Lip_{p^{-1}}(\PP^1(\QQ_p), \W(k))$. 

\end{proof}%\footnote{I merged our proves and changed all the $`d'$'s to $`e'$'s. Let me know what you think. -- LEM}

%\begin{proof}
%Let $d_1 \in \PP^1(\QQ_p)$. Successively choose $d_l,\ l \ge 2$ in the following manner: $d_2 \in \PP^1(\QQ_p)$ such that $d_{\PP^1(\QQ_p)}(d_1,d_2) = p^{-1}$, as paths in $\tree$, $d_1$ and $d_2$ will share exactly one common edge. Choose $d_3 \in \PP^1(\QQ_p)$ so that $d_{\PP^1(\QQ_p)}(d_2,d_3) = p^{-2}$, that is $d_2$ and $d_3$ share two edges while $d_1$ and $d_3$ share a single edge. In general choose $d_l \in \PP^1(\QQ_p)$ such that $d_{\PP^1(\QQ_p)}(d_{l-1},d_l) = p^{-l+1}$ which will force $d_{l-1}$ and $d_l$ to share $l-1$ edges as paths in $\tree$. Denote $D = \{d_l\}_{l \ge 1}$, which is a countable discrete set in $\PP^1(\QQ_p)$.
%
%We now define a family of admissible vectors, $\mbf{a}_i$, using the paths corresponding to the $\{d_l\}$. For each infinite rooted path in $\tree$, say $r$, give every vertex on $r$ the label zero if $d_{\PP^1(\QQ_p)}(r,d_l) = 1$ for all $l \ge 1$. For $r \in \PP^1(\QQ_p)$ such that $d_{\PP^1(\QQ_p)}(r,d_l) <1$ for some $l \ge 1$ let $n(p) = \max \{l \colon d_{\PP^1(\QQ_p)}(r,d_l) = p^{-l} \}$ then along each path $r$ label every vertex zero except for the one of weight $p^{l+i}$ which is labeled with a $1$. Since this is an admissible vector define $f_i = \Phi(\mbf{a}_i)$.
%\end{proof}

Now we are set to demonstrate the dimension calculation. 

\begin{theorem}
\label{thm:diminfinite}
The Krull dimension of $\W_{\ZZ_p^d}(k)$ is infinite for $d \geq 2$. 
\end{theorem}
\begin{proof}
It suffices to show that $\W_{\ZZ_p^2}(k)$ is infinite dimensional as $\W_{\ZZ_p^d}(k)$ surjects onto $\W_{\ZZ_p^2}(k)$ when $d \geq 2$, and so $\dim \W_{\ZZ_p^d}(k) \geq \dim \W_{\ZZ_p^2}(k)$. It also suffices to show that $R$ has infinite dimension. Since $R \cong \Lip_{p^{-1}}(\PP^1(\QQ_p), \W(k))$ by Theorem~\ref{thm:isomorphism} we chose to work in the latter ring. 

Construct $E = \{e_1, e_2, \ldots \}$ as in Lemma~\ref{lem:funzero} and pick $\cU$ any non-principal ultrafilter on $\PP^1(\QQ_p)$ containing $E$ (such an ultrafilter exists as we can start with any non-principal filter containing $E$ and construct an ultrafilter containing it). For $i \in \NN$ set $\sigma_i(e_n) = \omega_{p^{n^i}}(1)$. By Lemma~\ref{lem:ultraprime}, to each $\sigma_i$ we obtain a prime ideal $\fp_{\sigma_i}$. Assume $i < j$. We claim $\fp_{\sigma_j} \subset \fp_{\sigma_i}$.  For any $f \in \fp_{\sigma_j}$ there is $A \subset E$ and $c > 0$ such that for any $e_n \in A$ we have  $\ord(f(e_n)) \geq c \ord ( \sigma_j( e_n))$. By definitions, we have 
\begin{eqnarray*}
\ord(f(e_n)) &\geq & c \ord ( \sigma_j( e_n)) \\
 & = & cn^j \\
 & \geq & c n^i.
\end{eqnarray*} Thus we have $f \in \fp_{\sigma_i}$ and a chain of primes $$\fp_{\sigma_1} \supset \fp_{\sigma_2} \supset \cdots.$$ Use Lemma~\ref{lem:funzero} with $i \geq 2$ to pick $f_i \in \Lip_{p^{-1}}(\PP^1(\QQ_p), \W(k))$ such that $\ord f_i(e_n) = n^i$, see Remark~\ref{rmk:conversionlineartoexponential}. Therefore $f_i \in \fp_{\sigma_i}$ with the choice $A = E$ and $c = 1$ as in Lemma \ref{lem:ultraprime}.

We show that $f_i \not\in \fp_{\sigma_j}$ for $i < j$ by contradiction. Provided $f_i \in \fp_{\sigma_j}$, there is $A \in \cU$ with $A \subset E$ and $c > 0$ such that for any $e_n \in A$, $n^i = \ord ( f(e_n))  \geq c \ord (\sigma_j(e_n)) = c n^j$. This means $n^{i - j } \geq c$ and since $i < j$, there are only finitely many values $n$ for which this holds. Thus the set $A$ is finite as it contains only finitely many $e_n$. This contradicts the choice of $\cU$ as a  non-principal ultrafilter. Thus $\fp_{\sigma_j} \subsetneq \fp_{\sigma_i}$ for any $i \geq 2$ and $j > i$ and so $\dim \Lip_{p^{-1}}(\PP^1(\QQ_p), \W(k))$ is infinite.
% $n+i \leq c(n+j)$ which means $ n ( c - 1) \leq i - cj$. In the case that $c > 1$ we have $n \leq { i - cj \over c -1 }$ which now is either positive or negative. In either situation there is a contradiction as this would force $A$ to be either finite or empty; either case contradicts $\cU$ being a non-principal ultrafilter. In the case $c \leq 1$, we have $n + i \geq c( n+j )$ and is an integer, therefore $n + i \geq \lceil c (n + j) \rceil = n + j$ which gives $i \geq j$ which is again a contradiction. Thus $\fp_{\sigma_j} \subsetneq \fp_{\sigma_i}$ and so $\dim \Lip_{p^{-1}}(\PP^1(\QQ_p), \W(k))$  is infinite.
\end{proof}

%\begin{remark} Here is a different possible contradiction and likely better contradiction. This was the original one in the Osborne paper. Let  $\sigma_i(a_n) = \omega_{p^{n^i}}(1)$. By Lemma~\ref{lem:ultraprime}, to each $\sigma_i$ we obtain a prime ideal $\fp_{\sigma_i}$. Since $\ord \sigma_i(a_n) = -n^i$, when $i < j$ we have $-n^i > -n^j$. Therefore when $f \in \fp_{\sigma_j}$ we have a set $A \in \cU$ with $A \subset D$ and $c > 0$ such that for any $z \in A$, $\ord f(z) \leq c(-n^j) < c(-n^i)$ and so $f \in \fp_{\sigma_i}.$ Thus we have a chain of primes $$\fp_{\sigma_1} \supset \fp_{\sigma_2} \supset \cdots.$$
%
%Now assume $i < j$ and  $f \in \Lip_{p^{-1}}(\PP^1(\QQ_p), \W(k))$ with $\ord f(d_n) = -n^i$. Such a function exists by Lemma~\ref{lem:funzero}. By definition $f \in \fp_{\sigma_i}$ with the choice $A = D$ and $c = 1$. If $f \in \fp_{\sigma_j}$ then there is a set $A \subset D$ and $c > 0$ such that $\ord f(d_n) \leq c ( -n^j)$ for all $d_n \in A$. Therefore for all $d_n \in A$, we have $-n^i \leq c (-n^j)$, i.e., $n^{i - j} \geq c$ which can happen for only finitely many $n$ as $i - j < 0$ which is a contradiction as $\cU$ is a non-principal ultrafilter.
%\end{remark}

\subsection{Topological basis for $\Lip_{p^{-1}}(\PP^1(\QQ_p), \W(k))$}\label{sec:vdP}

%Here we discuss M\"ahler and van der Put basis for $\Lip_{1 \over p}(\PP^1(\QQ_p), W(k))$ and Teichm\"uller basis for $R$ and how they align. This needs to contain a review of the vdP basis from \cite{Sch06} for $C(\ZZ_p)$ and the extension of the Stone-Weierstrass theorem for compact non-archimedian metric spaces so that it extends to $C(\PP^1(\QQ_p))$. 

One of the main tools of studying non-archimedean function spaces, for example such as $\C(X,\ZZ_p)$ where $X$ is a topological space, is through topological bases which allow for a rich integration theory. Most familiar is the Mahler basis for $\C(\ZZ_p, \CC_p)$ or more generally $\C(\ZZ_p, K)$ where $K$ is a non-archimedean valued field. It consists of the functions $x \mapsto {x \choose n} = {{x (x-1) \cdots (x -n +1)} \over n!}$ for $n = 0, 1, \ldots$ which form an orthonormal basis for $\C(\ZZ_p, K)$. One might ask: Does $\Lip_{p^{-1}}(\PP^1(\QQ_p), \W(k))$ have a natural orthonormal basis? Since $\Lip_{p^{-1}}(\PP^1(\QQ_p), \W(k)) \cong R$ is isometric and $R$ has a particularly nice basis of \Teich elements, we know that $\Lip_{p^{-1}}(\PP^1(\QQ_p), \W(k))$ does as well.  This basis is unlike the Mahler basis because the basis elements are localized functions which makes them much more like the {\it van der Put basis} of $\C(\ZZ_p,K)$ \cite{Sch06}.
%The answer is it is sort of global (in that the domain is $\PP^1(\QQ_p)$) Lipschitz analogue of another orthonormal basis for $\C(\ZZ_p,K)$ called a {\it van der Put basis} \cite{Sch06}. 

For convenience of the reader we pause to review the van der Put bases in the more classic setting $\C(\ZZ_p,K)$ where $K$ is a non-archimedean valued field.  Let $x \in \ZZ_p$ and $m \in \ZZ$. We say that {\it $x$ starts with $m$} when $|x - m|_p < {1 \over m}$ and this relationship is denoted $m \lhd x$. A {\it van der Put function} is defined as $ e_n(x) = 1$ when $n \lhd x$ and is $0$ otherwise. By \cite[Thm. 62.2]{Sch06} the collection $\{e_n(x)\}_{n=0}^{\infty}$ is a basis of $\C(\ZZ_p,K)$ where $K$. Even better \cite[Thm. 63.2]{Sch06} gives a characterization of which elements of $\C(\ZZ_p,K)$ are in $\Lip_\alpha(\ZZ_p,K)$ for $\alpha > 0$ in terms of this basis. Note that the van der Put functions $e_n(x)$ have as their support the open ball $B(n,p^{-j})$ where $|n|_p = p^{-j}$; i.e., $e_n(x)$ is the indicator of a ball in $\ZZ_p$. Using this characterization of the van der Put basis as indicator functions we can extend the basis from functions on $\ZZ_p$ to functions on $\PP^1(\QQ_p)$ by the following.

\begin{definition}
A subset $V$ of $C(\PP^1(\QQ_p),\W(k))$ is a {\it van der Put} basis provided it consists of all indicator functions of balls of positive radius in $\PP^1(\QQ_p)$.
\end{definition}

These indicators are continuous functions on $\PP^1(\QQ_p)$ but their Lipschitz constants are unbounded. In order to fit them into $\Lip_{p^{-1}}$ they can be scaled to reduce their supremum norm until their Lipschitz constant is $p^{-1}$. This does require a different scalar multiple for each basis element, however the scalar depends only on the size of the support of the basis element.

\begin{theorem}
\label{thm:TeichvdP}
Denote by $B \subset \Lip_{p^{-1}}(\PP^1(\QQ_p),\W(k))$ the image under $\Phi$ of the Teichm\"uller basis for $R$. 
The set $B$ consists of elements of the
%\footnote{which? I think what you were trying to say in the earlier draft is that there is an analogue vdP for something like $\C(\ZZ_p, \W(k))$. This is what we should describe precisely before this theorem and state it as what you mean here. -- LEM} 
van der Put basis for $C(\PP^1(\QQ_p),\W(k))$ scaled to have Lipschitz constant $1/p$. 
\end{theorem}
\begin{proof}
Consider the part of the boundary of $\tree$ that consists of only $p$ of the $p+1$ subtrees coming from the root, $T_0$. This portion of the boundary is isometric to $\ZZ_p$. We may consider the results of \cite{Sch06} concerning $\C(\ZZ_p,K)$  where $K$ is the fraction field of $\W(k)$ as applying locally in our setting. Exercise 63.B of \cite{Sch06} says that $f \in \C(\ZZ_p,K)$ is actually in $\Lip_{\alpha}(\ZZ_p,K)$ provided 
$$ f (x)= \sum_{n=0}^{\infty} a_n e_n(x)$$
is the van der Put expansion of $f$ and the coefficients, $a_n$ satisfy
\begin{equation}
	\sup\{ |a_n|_K n^{1/\alpha} \} < \infty. \label{eq:LipnormvdP}
\end{equation}
As a consequence of Theorem \ref{thm:TeichvdP} the push forward of a Teichm\"uler expansion of an element of $R$ to a van der Put expansion of an element of $\C(\PP^1(\QQ_p),\W(k))$ can be restricted to a function over a copy of $\ZZ_p$. The coefficients of this expansion will be of the form $qp^{- \lfloor \log_p(n) \rfloor}$ for $q \in \FF_p$. These coefficients have $K$-norm bounded above by $n$ and there exists a subsequence, $n = p^{m}$, along which the norm is $n$. So for $\alpha \ge 1$ the supremum in (\ref{eq:LipnormvdP}) is finite. 
\end{proof}

\end{document}